\documentclass[12pt, a4paper]{amsart}

\usepackage[hmargin=30mm, vmargin=25mm, includefoot, twoside]{geometry}
\usepackage[bookmarksopen=true]{hyperref}

\usepackage{amscd}
\usepackage{amsfonts,amssymb,verbatim}
\usepackage{latexsym}
\usepackage{mathrsfs}
\usepackage{stmaryrd}
\usepackage{xspace}
\usepackage{enumerate, paralist}
\usepackage{graphicx}
\usepackage[all]{xy}
\usepackage{extarrows}

\usepackage[usenames,dvipsnames]{color}

\usepackage{txfonts, pxfonts}

\usepackage{amsthm}
\usepackage{amsmath}

\usepackage{todonotes}
\numberwithin{equation}{section}

\newtheorem{thm}{Theorem}[section]
 \newtheorem{cor}[thm]{Corollary}
 \newtheorem{lem}[thm]{Lemma}
 \newtheorem{prop}[thm]{Proposition}

\newtheorem{alphthm}{Theorem}			

\newtheorem{alphprop}[alphthm]{Proposition}

 \theoremstyle{definition}
  \newtheorem{defn}[thm]{Definition}
  
  \newtheorem{question}[thm]{Question}

\newtheorem{alphdefn}[alphthm]{Definition}			

\newtheorem{alphques}[alphthm]{Question}			

 \theoremstyle{remark}
 \newtheorem{rem}[thm]{Remark}
  \newtheorem{ex}[thm]{Example}

\newtheorem*{claim*}{Claim}

\def\NN{\mathbb{N}}
\def\RR{\mathbb{R}}
\def\CC{\mathbb{C}}

\def\H{\mathcal{H}}
\def\Nd{\mathcal{N}}
\def\S{\mathcal{S}}
\def\G{\mathcal{G}}

\def\P{\mathcal{P}}
\def\F{\mathcal{F}}
\def\M{\mathcal{M}}

\def\E{\mathscr{E}}
\def\L{\mathbf{L}}
\def\I{\mathscr{I}}

\def\B{\mathfrak{B}}
\def\K{\mathfrak{K}}

\def\supp{\mathrm{supp}}
\def\ppg{\mathrm{prop}}
\def\Id{\mathrm{Id}}

\def\max{\mathrm{max}}

\def\r{\mathrm{r}}
\def\s{\mathrm{s}}

\def\Gr{\mathrm{Gr}}
\def\SOT{\mathrm{SOT}}

\begin{document}

\title{Rigidity for geometric ideals in uniform Roe algebras}

\author{Baojie Jiang and Jiawen Zhang}

\address[Baojie Jiang]{School of Mathematical Sciences, Chongqing Normal University, Chongqing 401331, China.}
\email{jiangbaojie@cqnu.edu.cn}

\address[Jiawen Zhang]{School of Mathematical Sciences, Fudan University, 220 Handan Road, Shanghai, 200433, China.}
\email{jiawenzhang@fudan.edu.cn}

\date{}


\begin{abstract}
In this paper, we investigate the rigidity problems for geometric ideals in uniform Roe algebras associated to discrete metric spaces of bounded geometry. These ideals were introduced by Chen and Wang, and can be fully characterised in terms of ideals in the associated coarse structures. Our main result is that if two geometric ideals in uniform Roe algebras are stably isomorphic, then the coarse spaces associated to these ideals are coarsely equivalent. We also discuss the case of ghostly ideals and pose some open questions.
\end{abstract}

\date{\today}
\maketitle

\parskip 4pt

\noindent\textit{Mathematics Subject Classification} (2020): 47L20, 46L80, 51F30, 47L40.\\
\textit{Keywords: Uniform Roe algebras, Geometric ideals, Rigidity, Coarse equivalences}

\section{Introduction}\label{sec:intro}

Roe algebras are $C^*$-algebras associated to metric spaces, which encode the coarse geometry of the underlying spaces. They were introduced by Roe in his pioneering work \cite{Roe88} to study higher indices of differential operators on open manifolds. There is also a uniform version of the Roe algebra, which has found applications in index theory (\emph{e.g.},\cite{Spa09}), $C^*$-algebra theory (\emph{e.g.}, \cite{Oza00}), single operator theory (\emph{e.g.}, \cite{SW17}) and even mathematical physics (\emph{e.g.}, \cite{EM19}).

To provide a formal definition, consider a discrete metric space $(X,d)$ of bounded geometry (see Section \ref{ssec:notions from coarse geometry}). Thinking of an operator $T$ on $\ell^2(X)$ as an $X$-by-$X$ matrix $(T(x,y))_{x,y \in X}$, we say that $T$ has \emph{finite propagation} if $\sup\{d(x,y) ~|~ T(x,y) \neq 0\} < \infty$. The set of all finite propagation operators forms a $\ast$-subalgebra of $\B(\ell^2(X))$, and its norm closure is called the \emph{uniform Roe algebra of $X$} and denoted by $C^*_u(X)$.

Uniform Roe algebras have nice behaviour in coarse geometry. More precisely, recall that two metric spaces $(X,d_X)$ and $(Y,d_Y)$ are \emph{coarsely equivalent} if there exist a map $f: X \to Y$ and functions $\rho_+, \rho_-: [0,\infty) \to \RR$ with $\lim\limits_{t\to +\infty} \rho_{\pm}(t) = +\infty$ such that for any $x_1,x_2 \in X$, we have
\[
\rho_-(d_X(x_1, x_2)) \leq d_Y(f(x_1), f(x_2)) \leq \rho_+(d_X(x_1, x_2)),
\]
and there exists $R>0$ such that the $R$-neighbourhood of $f(X)$ equals $Y$. It is known (see, \emph{e.g.}, \cite[Theorem 4]{BNW07}) that if $(X,d_X)$ and $(Y,d_Y)$ are coarsely equivalent, then their uniform Roe algebras are \emph{stably isomorphic} in the sense that $C^*_u(X) \otimes \K(\H) \cong C^*_u(Y) \otimes \K(\H)$ (where $\K(\H)$ is the $C^*$-algebra of compact operators on a separable infinite dimensional Hilbert space $\H$).

Conversely, the rigidity problem concerns whether the coarse geometry of a metric space can be fully determined by the associated uniform Roe algebra, \emph{i.e.}, whether two metric spaces are coarsely equivalent if their uniform Roe algebras are (stably) isomorphic. This problem was initially studied by \v{S}pakula and Willett in \cite{SW13}, followed by a series of works in the last decade (\emph{e.g.}, \cite{BCL20, BF21, LSZ20}), and recently is completely solved by the profound work \cite{BBFKVW22}.

Due to the importance of uniform Roe algebras, Chen and Wang initiated the study of their ideal structures (\cite{CW04, CW05, Wan07}). They managed to obtain a full description for the ideal structure of the uniform Roe algebra when the underlying space has Yu's Property A (from \cite{Yu00}). More precisely, they introduced a notion of geometric ideal and provided a detailed picture for these ideals.
Let us recall the definition:

\begin{alphdefn}[\cite{CW04, Wan07}]\label{introdefn:geom ideal}
An ideal $I$ in the uniform Roe algebra $C^*_u(X)$ of a discrete metric space $(X,d)$ of bounded geometry is called \emph{geometric} if the set of finite propagation operators in $I$ is dense in $I$.
\end{alphdefn}

To state the characterisations for geometric ideals in \cite{CW04}, it is convenient to consult the notion of coarse spaces (see, \emph{e.g.}, \cite{Roe93, Roe03}). Recall that a \emph{coarse structure} on a set $X$ is a collection $\E$ of subsets of $X \times X$ which is closed under the formation of subsets, inverses, products and finite unions (see Definition \ref{defn:coarse structure}). For $(X,\E)$, we can also define the uniform Roe algebra $C^*_u(X,\E)$ similarly as in the metric space case while replacing finite propagation operators with those whose supports belong to $\E$ (see Definition \ref{defn: unif. Roe alg.}). In the case of a metric space $(X,d)$, there is an associated coarse structure $\E_d$ which is the smallest one containing the sets $E_R:=\{(x,y) ~|~ d(x,y) \leq R\}$ for all $R\geq  0$, and clearly we have $C^*_u(X,\E_d)= C^*_u(X)$.

In \cite{CW04}, the authors discovered that geometric ideals in $C^*_u(X)$ for a metric space $(X,d)$ can be described in terms of ideals in the coarse structure $\E_d$. Recall that an \emph{ideal in $\E_d$} is a coarse structure $\I \subseteq \E_d$ on $X$ which is closed under products by elements in $\E_d$ (see Definition \ref{defn:ideals in coarse structure}). The geometric ideal in $C^*_u(X)$ associated to $\I$ is $C^*_u(X,\I)$, and this procedure provides an isomorphism between the lattice of geometric ideals in $C^*_u(X)$ and the lattice of ideals in $\E_d$ (\cite[Theorem 6.3]{CW04}). For convenience, we denote $\I(I)$ the ideal in $\E_d$ associated to a geometric ideal $I$ in $C^*_u(X)$. Moreover, \cite[Theorem 6.3]{CW04} also shows that ideals in $\E_d$ can be described using ideals in $(X,d)$ (see Definition \ref{defn:ideals in space}). The ideal in $(X,d)$ associated to an ideal $\I$ in $\E_d$ is $\L(\I):=\{\r(E) ~|~ E \in \I\}$, where $\r: X \times X \to X$ is the projection onto the first coordinate.

The main focus of this paper is to study the rigidity problem for geometric ideals in uniform Roe algebras of metric spaces. More precisely, we ask the following:

\begin{alphques}[Rigidity for geometric ideals]\label{introques}
Let $(X,d_X), (Y,d_Y)$ be discrete metric spaces of bounded geometry, and $I_X, I_Y$ be geometric ideals in the uniform Roe algebras $C^*_u(X)$ and $C^*_u(Y)$, respectively. If $I_X$ and $I_Y$ are (stably) isomorphic, do $(X,\I(I_X))$ and $(Y,\I(I_Y))$ have the same structure?
\end{alphques}

Our first task is to make the phrase \emph{``have the same structure''} in a more precise way. Readers might wonder whether it is possible to use a similar notion of coarse equivalence as in the metric space case recalled above. This works well if the coarse structures contain the diagonals (see, \emph{e.g.}, Definition \ref{defn:coarse equivalence} and the paragraph thereafter). (Note that in \cite[Definition 2.3]{Roe03}, one requires that a coarse structure always contains the diagonal. While in \cite{CW04,STY02}, it is necessary to consider the more general case.) However in the general case, there would be some issue if we only consider a single map (as in Definition \ref{defn:coarse equivalence}) due to the lack of units in geometric ideals. More precisely, note that a nontrivial geometric ideal $I$ in $C^*_u(X)$ does not bare a unit, and hence the associated ideal $\I(I)$ does not contain the diagonal.


To overcome this issue, we need to consult the notion of coarse equivalence for general coarse spaces introduced by Skandalis, Tu and Yu:

\begin{alphdefn}[{\cite[Definition 2.2]{STY02}}]\label{introdefn:ce}
Let $(X, \E_X)$ and $(Y, \E_Y)$ be coarse spaces. A \emph{coarse correspondence from $(X, \E_X)$ to $(Y, \E_Y)$} is a coarse structure $\E$ on $X \sqcup Y$ which restricts to $\E_Y$ on $Y$, contains $\E_X$, and is generated by the elements contained in $Y \times (X \sqcup Y)$. A \emph{coarse equivalence} between $(X, \E_X)$ and $(Y, \E_Y)$ is a coarse structure on $X \sqcup Y$ which is a coarse correspondence from $X$ to $Y$ and from $Y$ to $X$.
\end{alphdefn}

As noted in \cite[Proposition 2.3]{STY02} (see also Proposition \ref{prop:coarse corr unital case} and Corollary \ref{cor:coarse equiv unital case}), Definition \ref{introdefn:ce} coincides with the notion of coarse equivalence recalled above in the metric space case. However, it seems inconvenient to use directly the language of coarse correspondence to treat the rigidity problem. Hence we unpack Definition \ref{introdefn:ce} by means of families of maps and prove the following:

\begin{alphprop}[Corollary \ref{cor:coarse equiv general case}]\label{introprop:char for ce}
Two coarse spaces $(X, \E_X)$ and $(Y, \E_Y)$ are coarsely equivalent \emph{if and only if} there exist coarse families (see Definition \ref{defn:coarse family}) of maps
\[
\F=\{f_L:L \to Y ~|~ L \in \L(\E_X)\} \quad \text{and} \quad \G=\{g_{L'}: L' \to X ~|~ L' \in \L(\E_Y)\}
\]
such that $\{(g_{f_L(L)} \circ f_L(x), x) ~|~ x\in L\} \in \E_X$ and $\{(f_{g_{L'}(L')} \circ g_{L'}(y), y) ~|~ y\in L'\} \in \E_Y$ for any $L \in \L(\E_X)$ and $L' \in \L(\E_Y)$.
\end{alphprop}

Having established Proposition \ref{introprop:char for ce}, we manage to answer Question \ref{introques} completely. The following is the main result of this paper (see Section \ref{sec:morita for geometric} and Section \ref{sec:stable} for the missing definitions):

\begin{alphthm}\label{introthm}
Let $(X,d_X)$ and $(Y,d_Y)$ be discrete metric spaces of bounded geometry, and $I_X, I_Y$ be geometric ideals in the uniform Roe algebras $C^*_u(X)$ and $C^*_u(Y)$, respectively. Consider the following conditions:
\begin{enumerate}
 \item $I_X$ and $I_Y$ are stably isomorphic;
 \item $(X,\I(I_X))$ and $(Y, \I(I_Y))$ are coarsely equivalent;
 \item $I_X$ and $I_Y$ are Morita equivalent.
\end{enumerate}
Then we have (1) $\Rightarrow$ (2) $\Rightarrow$ (3). Additionally if the associated ideals $\L(\I(I_X))$ and $\L(\I(I_Y))$ are countably generated, then the three conditions above are all equivalent.
\end{alphthm}

Note that ``(2) $\Rightarrow$ (3)'' is known to experts, and the key step is to prove ``(1) $\Rightarrow$ (2)''. The strategy  follows from that for \cite[Theorem 1.4]{BBFKVW22}, which is the standard approach originated in \cite{SW13}. However, there are technical issues due to the lack of units in geometric ideals. To overcome this, we discover concrete approximate units in geometric ideals (see Lemma \ref{lem:app unit}), which have close relations to the coarse geometry of the underlying spaces and the associated ideals. Hence we can approximate a geometric ideal by a net of uniform Roe algebras of subspaces, and therefore manage to consult the proof in the unital case.

At this point, we would also like to highlight that Theorem \ref{introthm} is \emph{not} just an easy combination of results in \cite[Section 4]{BF21} and \cite{BBFKVW22}. Recall that in \cite{BF21}, the rigidity problem was studied for general coarse spaces. Note that there is an extra hypothesis that coarse spaces are \emph{small} in the sense of \cite[Definition 4.2]{BF21} to prove that rigid isomorphisms induce coarse equivalences. However in the current setting, coarse structures associated to geometric ideals might not be small in general. This obstructs us from directly using the techniques from \cite[Section 4]{BF21}. To overcome the issue, again we make use of the concrete approximate units to reduce to the unital case.

Also note that in a recent work \cite{WZ23}, Wang and the second-named author introduced a notion of ghostly ideals and studied the ideal structure of uniform Roe algebras beyond the scope of Property A. We provide some discussions on the rigidity problems for ghostly ideals and pose some open questions.

The paper is organised as follows. In Section \ref{sec:pre}, we recall necessary background knowledge in coarse geometry and uniform Roe algebras. In Section \ref{sec:coarse corr}, we study the notion of coarse equivalence for general coarse spaces (Definition \ref{introdefn:ce}) and prove Proposition \ref{introprop:char for ce}. Section \ref{sec:morita for geometric} to Section \ref{sec:stable} are devoted to the proof for Theorem \ref{introthm}. We record the proof for the easy direction ``(2) $\Rightarrow$ (3)'' in Section \ref{sec:morita for geometric} and prove ``(1) $\Rightarrow$ (2)'' in Section \ref{sec:stable}. To make the argument in Section \ref{sec:stable} more transparent, we prove a weak version of ``(1) $\Rightarrow$ (2)'' in Section \ref{sec:rigidity for geometric ideals}. Finally, we provide some discussion on ghostly ideals in Section \ref{sec:ghostly ideals}.

\subsection*{Acknowledgement} This project began during BJ's visit to the School of Mathematical Sciences at Fudan University. BJ would like to express his gratitude to Prof. Xiaoman Chen and Yijun Yao for the invitation and coordination, and also to the faculty members for their hospitality and engaging discussions during his stay. We would also like to thank Rufus Willett for his comments after reading an early draft of this paper.

BJ was supported by NSFC12001066 and NSFC12071183. JZ was partly supported by National Key R{\&}D Program of China 2022YFA100700.

\section{Preliminaries}\label{sec:pre}

\subsection{Standard notation} Here we collect the notation used throughout the paper.

For a set $X$, denote $|X|$ the cardinality of $X$ and $\P(X)$ the set of all subsets in $X$. For $A \subseteq X$, denote $\chi_{A}$ the characteristic function of $A$, and set $\delta_x:=\chi_{\{x\}}$ for $x\in X$.

For a discrete space $X$, denote $\ell^\infty(X)$ the $C^*$-algebra of bounded functions on $X$ with the supremum norm $\|f\|_\infty:=\sup_{x\in X}|f(x)|$. The \textit{support} of $f\in \ell^\infty(X)$ is defined to be $\{ x\in X ~|~ f(x)\neq 0\}$, denoted by $\supp(f)$. Given a Hilbert space $\H$, denote $\ell^2(X;\H)$ the Hilbert space of square-summable functions from $X$ to $\H$. When $\H = \CC$, we denote $\ell^2(X):=\ell^2(X;\CC)$, which has an orthonormal basis $\{\delta_x\}_{x\in X}$.

Given a Hilbert space $\H$, denote $\B(\H)$ the $C^*$-algebra of all bounded linear operators on $\H$, and $\K(\H)$ the $C^*$-subalgebra of all compact operators on $\H$.

\subsection{Notions from coarse geometry}\label{ssec:notions from coarse geometry}

Here we collect necessary notions from coarse geometry, and guide readers to \cite{NY12, Roe03} for more details.

First recall that for a set $X$ and $E,F \subseteq X \times X$, denote
\begin{align*}
E^{-1}&:=\{(y,x) ~|~ (x,y) \in E\},\\
E \circ F &:= \{(x,z) ~|~ \exists~ y\in X \text{ such that } (x,y) \in E \text{ and } (y,z) \in F\}.
\end{align*}
Also denote $\r,\s: X\times X \to X$ the projection onto the first and the second coordinate, respectively. Given $E \subseteq X \times X$ and $A,B \subseteq X$, we say that $A$ and $B$ are \emph{$E$-separated} if $(A \times B) \cap E = \emptyset$ and $(B \times A) \cap E = \emptyset$.

\begin{defn}\label{defn:coarse structure}
A \emph{(connected) coarse structure} on a set $X$ is a collection $\E \subseteq \P(X \times X)$, called \emph{entourages}, satisfying the following:
\begin{enumerate}
 \item[(1)] For any entourages $A$ and $B$, then $A^{-1}$, $A \circ B$ and $A \cup B$ are entourages;
 \item[(2)] Every finite subset of $X \times X$ is an entourage;
 \item[(3)] Any subset of an entourage is an entourage.
\end{enumerate}
In this case, $(X, \E)$ is called a \emph{coarse space}. If additionally the diagonal
\[
\Delta_X:=\{(x,x) ~|~ x\in X\}
\]
is an entourage, then $\E$ (also the pair $(X, \E)$) is called \emph{unital}.
\end{defn}

For $Y \subseteq X$ and $E \in \E$, denote the \emph{$E$-neighbourhood} $\Nd_E(Y)$ of $Y$ by
\[
\Nd_E(Y):=Y \cup \{x\in X ~|~\exists~y\in Y \text{ such that } (x,y) \in E\}
\]
and
\[
n(E) := \sup_{x\in X} \max\{~ |\Nd_E(\{x\})|, |\Nd_{E^{-1}}(\{x\})|~\}.
\]

\begin{defn}\label{defn:bdd geom for coarse structure}
A coarse structure $\E$ on a set $X$ is said to have \emph{bounded geometry} (or to be \emph{uniformly locally finite}) if $n(E)$ is finite for any entourage $E \in \E$. In this case, we also say that the coarse space $(X,\E)$ has bounded geometry.
\end{defn}

For a set $X$ and a collection $\S\subseteq \P(X \times X)$, the smallest coarse structure on $X$ containing $\S$ is called the coarse structure \emph{generated by $\S$}.

When $(X,d)$ is a discrete metric space, there is an associated unital coarse structure $\E_d$ (called the \emph{bounded coarse structure}) generated by all the \emph{$R$-entourages} defined as $E_R :=\{(x,y) \in X \times X ~|~ d(x,y) \leq R\}$ for all $R\geq 0$. In this case, we denote the closed ball by $B_{X}(x,r):=\Nd_{E_r}(\{x\})$ for $x\in X$ and $r\geq 0$, and $\Nd_r(Y):=\Nd_{E_r}(Y)$ for $Y\subseteq X$ and $r\geq 0$. We say that $(X,d)$ has \emph{bounded geometry} if $\E_d$ has bounded geometry, \emph{i.e.}, the number $\sup_{x\in X} |B_{X}(x,r)|$ is finite for any $r\geq 0$.

\begin{defn}\label{defn:coarse maps}
Let $X,Y$ be sets, $f: X \to Y$ be a map and $\E_Y$ be a coarse structure on $Y$.
\begin{enumerate}
 \item[(1)] Denote $f^*\E_Y:=\{E \subseteq X \times X ~|~ (f \times f)(E) \in \E_Y\}$, which is a coarse structure on $X$. Here $(f \times f)(E):=\{(f(x),f(y)) ~|~ (x,y) \in E\}$.
 \item[(2)] If $\E_X$ is a coarse structure on $X$, we say that $f$ is \emph{coarse} if $\E_X \subseteq f^*\E_Y$. To specify the underlying coarse structures, we also write $f: (X, \E_X) \to (Y, \E_Y)$.
 \item[(3)] Two coarse maps $f,g: (X, \E_X) \to (Y, \E_Y)$ are said to be \emph{close} if for any $E \in \E_X$, we have $(f \times g)(E):=\{(f(x),g(y)) ~|~ (x,y) \in E\} \in \E_Y$.
\end{enumerate}
\end{defn}

\begin{defn}\label{defn:coarse equivalence}
Let $f: (X,\E_X) \to (Y, \E_Y)$ be a map between unital coarse spaces. We say that $f$ is a \emph{coarse equivalence} if $f$ is coarse and there exists a coarse map $g: (Y, \E_Y) \to (X, \E_X)$ (called a \emph{coarse inverse} to $f$) such that $f\circ g$ is close to $\Id_Y$ and $g \circ f$ is close to $\Id_X$. In this case, we say that $(X,\E_X)$ and $(Y, \E_Y)$ are \emph{coarsely equivalent}.
\end{defn}


When $(X,\E_X)$ and $(Y, \E_Y)$ come from metric spaces, Definition \ref{defn:coarse equivalence} coincides with the one recalled in Section \ref{sec:intro}. Note that Definition \ref{defn:coarse equivalence} also makes sense for the general (non-unital) case, however, it does \emph{not} work well for the rigidity problems (see Proposition \ref{introprop:char for ce} and Theorem \ref{introthm}). It turns out that the suitable setting for the notion of coarse equivalence in the general case is Definition \ref{introdefn:ce}. To obtain an appropriate picture, we will focus on Definition \ref{introdefn:ce} in Section \ref{sec:coarse corr} based on an alternative version of coarse maps from \cite{STY02}.

\subsection{Uniform Roe algebras and geometric ideals}\label{ssec:uniform Roe alg}
Let $X$ be a set. Each operator $T \in \B(\ell^2(X))$ can be written in the matrix form $T=(T(x,y))_{x,y\in X}$, where $T(x,y)=\langle T \delta_y, \delta_x \rangle \in \CC$.
Denote by $\|T\|$ the operator norm of $T$ in $\B(\ell^2(X))$.
Similarly for an operator $T \in \B(\ell^2(X;\H))$, we can also write $T=(T(x,y))_{x,y\in X}$ for $T(x,y) \in \B(\H)$.

Given an operator $T \in \B(\ell^2(X))$, we define the \emph{support} of $T$ to be
\[
\supp(T):=\{(x,y) \in X \times X ~|~ T(x,y) \neq 0\}.
\]
Given $\varepsilon>0$, define the \emph{$\varepsilon$-support of $T$} to be
\[
\supp_\varepsilon(T):=\big\{(x,y)\in X \times X ~\big|~ |T(x,y)| \geq \varepsilon\big\}.
\]
When $X$ is equipped with a metric $d$, we define the \emph{propagation} of $T$ to be
\[
\ppg(T):= \sup\{d(x,y) ~|~ (x,y) \in \supp(T)\}.
\]

\begin{defn}\label{defn: unif. Roe alg.}
Let $(X,\E)$ be a coarse space of bounded geometry.
The set of all operators in $\B(\ell^2(X))$ whose supports belong to $\E$ forms a $\ast$-algebra, called the \emph{algebraic uniform Roe algebra of $(X, \E)$} and denoted by $\CC_u[X,\E]$. The \emph{uniform Roe algebra of $(X, \E)$} is defined to be the operator norm closure of $\CC_u[X,\E]$ in $\B(\ell^2(X))$, which forms a $C^*$-algebra and is denoted by $C^*_u(X, \E)$.

When $(X,d)$ is a discrete metric space of bounded geometry, we simply write $\CC_u[X]:=\CC_u[X,\E_d]$ and $C^*_u(X):=C^*_u(X,\E_d)$.
\end{defn}

For a coarse space $(X,\E)$ of bounded geometry and $x,y\in X$, denote the rank-one operator $\xi \mapsto \langle \xi, \delta_y \rangle \delta_x$ by $e_{xy}$. It is clear that $e_{xy} \in \CC_u[X,\E]$. This kind of operators will play an important role when we study the rigidity problems later.

In \cite{CW04, Wan07}, Chen and Wang introduced the notion of geometric ideals (see Definition \ref{introdefn:geom ideal}) in uniform Roe algebras and provided a full description for these ideals. Here we only focus on the case of metric spaces, which are the main objects we are interested in for the rigidity problem. (Note that geometric ideals for general coarse spaces were also studied in \cite{CW04}.)
In the following, we always assume that $(X,d)$ is a discrete metric space of bounded geometry and $\E_d$ is the associated bounded coarse structure.

\begin{defn}[{\cite[Definition 4.1]{CW04}}]\label{defn:ideals in coarse structure}
A coarse structure $\I \subseteq \E_d$ is called an \emph{ideal}
in $\E_d$ if for any $E \in \E_d$ and $A \in \I$, we have $E \circ A \in \I$ and $A \circ E \in \I$.
\end{defn}

\begin{prop}[{\cite[Proposition 4.2]{CW04}}]\label{prop:ideals in coarse strcutrue}
Given an ideal $\I$ in $\E_d$, the uniform Roe algebra $C^*_u(X,\I)$ is a geometric ideal in $C^*_u(X)$. Conversely, given a geometric ideal $I$ in $C^*_u(X)$, the collection $\I(I):=\{\supp_\varepsilon(T) ~|~ T \in I, \varepsilon>0\}$ is an ideal in $\E_d$. Moreover, we have $\I(C^*_u(X,\I)) = \I$ and $C^*_u(X,\I(I)) = I$.
\end{prop}


We also need the following notion of ideals in space:

\begin{defn}[{\cite[Definition 6.1]{CW04}}]\label{defn:ideals in space}
An \emph{ideal} in $(X,d)$ is a collection $\L\subseteq \P(X)$ satisfying the following:
\begin{enumerate}
 \item If $Y \in \L$ and $Z \subseteq Y$, then $Z \in \L$;
 \item If $Y \in \L$ and $r>0$, then $\Nd_r(Y) \in \L$;
 \item If $Y ,Z \in \L$, then $Y \cup Z \in \L$.
\end{enumerate}
For $\S\subseteq \P(X)$, we say that $\L$ is the \emph{ideal generated by $\S$} if $\L$ is the smallest ideal in $(X,d)$ containing $\S$.
\end{defn}


\begin{prop}[{\cite[Proposition 6.2]{CW04}}]\label{prop:ideals in space}
Given an ideal $\I$ in $\E_d$, the collection $\L(\I):=\{\r(E) ~|~ E \in \I\}$ is an ideal in $(X,d)$. Conversely, given an ideal $\L$ in $(X,d)$, the collection $\I(\L):=\{E \in \E_d ~|~\exists~L \in \L \text{ such that } E \subseteq L \times L\}$ is an ideal in $\E_d$. Moreover, we have $\I(\L(\I)) = \I$ and $\L(\I(\L)) = \L$.
\end{prop}


Combining Proposition \ref{prop:ideals in coarse strcutrue} and Proposition \ref{prop:ideals in space}, we know that the lattice of geometric ideals in $C^*_u(X)$ is isomorphic to the lattice of ideals in $\E_d$, which is also isomorphic to the lattice of ideals in $(X,d)$. Hence we will drift among these three objects freely in the sequel.

\begin{ex}\label{ex:compact operators}
As a special case, we consider the ideal $\K(\ell^2(X))$ of compact operators in the uniform Roe algebra $C^*_u(X)$ for a discrete metric space $(X,d)$. Note that $\K(\ell^2(X))$ is a geometric ideal, and it is easy to see that $\I(\K(\ell^2(X)))$ consists of all finite subsets of $X \times X$, and $\L(\K(\ell^2(X)))$ consists of all finite subsets of $X$.
\end{ex}

\begin{rem}
We remark that the definition $\I(I)$ works for any (\emph{not necessarily geometric}) ideal (see \cite{CW04}) in the uniform Roe algebra $C^*_u(X)$. However, the one-to-one correspondence in Proposition \ref{prop:ideals in coarse strcutrue} only works for geometric ideals. As an easy example, one can consider the ideal $I_G$ consisting of all ghost operators in $C^*_u(X)$. (Recall that an operator $T \in \B(\ell^2(X))$ is a \emph{ghost operator} if for any $\varepsilon>0$, there exists a finite subset $F \subseteq X$ such that for any $(x,y) \notin F \times F$, then $|T(x,y)| < \varepsilon$.) One can calculate directly that $\I(I_G) = \I(\K(\ell^2(X)))$, which consists of all finite subsets in $X \times X$.
\end{rem}

Finally, we record an elementary observation which will be used later.

\begin{lem}\label{lem:app unit}
Let $\L$ be an ideal in $(X,d)$. Then the family $\{
\chi_L ~|~ L \in \L\}$ is an approximate unit for the ideal $C^*_u(X, \I(\L))$.
\end{lem}

\begin{proof}
Note that $C^*_u(X, \I(\L))$ is the direct limit of $\{C^*_u(L) ~|~L \in \L\}$ where $C^*_u(L)$ is the uniform Roe algebra of the metric space $L$ with the induced metric of $d$ and $\L$ is a direct set under inclusion. Since $\chi_L$ is the unit of $C^*_u(L)$, we conclude the proof.
\end{proof}

\section{Coarse correspondence}\label{sec:coarse corr}

In this section, we study the notion of coarse correspondence (Definition \ref{introdefn:ce}) from \cite{STY02} and provide a detailed picture for coarse equivalences between general coarse spaces (Proposition \ref{introprop:char for ce}).

Let $X$ be a set and $\E \subseteq \P(X \times X)$. For $A,B \subseteq X$, denote the \emph{restriction of $\E$ on $A \times B$} by
\[
\E|_{A\times B}:=\{E \cap (A \times B) ~|~ E \in \E\}.
\]
Recall from Definition \ref{introdefn:ce} that for coarse spaces $(X, \E_X)$ and $(Y, \E_Y)$, a \emph{coarse correspondence from $(X, \E_X)$ to $(Y, \E_Y)$} (or simply \emph{from $X$ to $Y$}) is a coarse structure $\E$ on the disjoint union $X \sqcup Y$ satisfying the following:
\begin{enumerate}
 \item $\E|_{Y \times Y} = \E_Y$;
 \item $\E|_{X \times X} \supseteq \E_X$;
 \item $\E$ is generated by $\E|_{Y \times (X \sqcup Y)}$.
\end{enumerate}


The following fact was implicitly stated in the proof of \cite[Proposition 2.3]{STY02}. It follows directly from condition (3) above, and hence we omit the proof.

\begin{lem}\label{lem:elementary facts for coarse corr}
Let $(X, \E_X)$ and $(Y, \E_Y)$ be coarse spaces, and $\E$ be a coarse correspondence from $(X, \E_X)$ to $(Y, \E_Y)$. Then for any $E \in \E|_{X \times X}$, there exists $F \in \E|_{Y \times X}$ such that $E \subseteq F^{-1} \circ F$.
\end{lem}

We also record the following elementary fact, whose proof is straightforward.

\begin{lem}\label{lem:easy containment}
For a set $X$ and $E \subseteq X \times X$, we have $E \subseteq (E')^{-1}\circ E'$ for $E':=E \cup (E \circ E^{-1})$.
\end{lem}

\subsection{The unital case}

It was shown in \cite{STY02} that for unital coarse spaces, the notion of coarse correspondence can be identified with coarse maps in Definition \ref{defn:coarse maps}(2). Here we recall the outline of the proof, and divide it into several pieces to clarify the dependence on the assumption of unitalness. Some of the results will also be used in the general case.

Let $X$ be a set, $(Y, \E_Y)$ be a coarse space and $f: X \to Y$ be a map. Denote
\[
\E_{YX}(f):=\{E \subseteq Y \times X ~|~ (\Id_Y \times f)(E) \in \E_Y\},
\]
and
\[
\E(f):=\{E_X \cup E_Y \cup E_{XY} \cup E_{YX} ~|~ E_X \in f^*\E_Y, E_Y \in \E_Y \text{ and } E_{YX} \cup E_{XY}^{-1} \in \E_{YX}(f)\}.
\]

The following result is contained in the proof of \cite[Proposition 2.3]{STY02}. For the reader's convenience, here we recall the proof.

\begin{lem}\label{lem:facts for E(f)}
Let $X$ be a set, $(Y, \E_Y)$ be a coarse space and $f: X \to Y$ be a map. Then $\E(f)$ is a coarse structure on $X \sqcup Y$ and generated by $\E(f)|_{Y \times (X \sqcup Y)} = \E_Y \cup \E_{YX}(f)$.
\end{lem}

\begin{proof}
Consider the map $h: X \sqcup Y \to Y$ which restricts to $f$ on $X$ and $\Id_Y$ on $Y$. It is easy to see that $\E(f) = h^*\E_Y$, and hence $\E(f)$ is a coarse structure on $X \sqcup Y$. For the second statement, it suffices to show that $\E(f)|_{X \times X} = f^*\E_Y$ can be generated by $\E(f)|_{Y \times X} = \E_{YX}(f)$. Given $E \in f^*\E_Y$, it follows from Lemma \ref{lem:easy containment} that $E \subseteq (E')^{-1}\circ E'$ for $E':=E \cup (E \circ E^{-1})$. It is clear that $E' \in f^*\E_Y$, and hence $F:=(f \times \Id_X)(E') \in \E_{YX}(f)$. Finally, note that $E \subseteq (E')^{-1}\circ E' \subseteq F^{-1} \circ F$, which concludes the proof.
\end{proof}

Consequently, we have the following:
\begin{cor}\label{cor:E(f) is coarse iff f is coarse}
Let $(X, \E_X)$ and $(Y, \E_Y)$ be coarse spaces, and $f: X \to Y$ be a map. Then $f$ is coarse \emph{if and only if} $\E(f)$ is a coarse correspondence. In this case, we say that $\E(f)$ is the \emph{coarse correspondence associated to $f$}.
\end{cor}

Recall that for a map $f: X \to Y$, the \emph{graph of $f$} is defined to be
\[
\Gr(f):=\{(f(x),x) \in Y \times X ~|~ x\in X\}.
\]
We have the following:

\begin{lem}\label{lem:E(f) is generated by graphs}
Let $X$ be a set, $(Y, \E_Y)$ be a coarse space and $f: X \to Y$ be a map. Assume that the set $\{(f(x),f(x)) \in Y \times Y ~|~ x\in X\}$ belongs to $\E_Y$. Then $\E(f)$ is generated by $\E_Y$ and $\Gr(f)$.
\end{lem}

\begin{proof}
Since $\{(f(x),f(x)) \in Y \times Y ~|~ x\in X\} \in \E_Y$, we have $\Gr(f) \in \E_{YX}(f)$. On the other hand, given $E\in \E_{YX}(f)$, we have $E \subseteq (E \circ \Gr(f)^{-1}) \circ \Gr(f)$. Since $E \circ \Gr(f)^{-1} = (\Id_Y \times f)(E) \in \E_Y$, we conclude the proof thanks to Lemma \ref{lem:facts for E(f)}.
\end{proof}

Now we present the following result from \cite{STY02} that for unital coarse spaces, each coarse correspondence is determined by a coarse map. For the reader's convenience, here we recall the proof.

\begin{prop}[{\cite[Proposition 2.3]{STY02}}]\label{prop:coarse corr unital case}
Let $(X, \E_X)$ and $(Y, \E_Y)$ be coarse spaces, and $\E$ be a coarse correspondence. Assume that $\E_X$ is unital. Then there exists a unique (up to closeness) coarse map $f: X \to Y$ such that $\E=\E(f)$.
\end{prop}

\begin{proof}
By Lemma \ref{lem:elementary facts for coarse corr} and the assumption that $\E_X$ is unital, there exists $E \in \E|_{Y \times X}$ such that $\Delta_X \subseteq E^{-1} \circ E$. It is easy to see that $E$ contains $\Gr(f)$ for some map $f: X \to Y$. Note that $\{(f(x),f(x)) \in Y \times Y ~|~ x\in X\} \subseteq \Gr(f) \circ \Gr(f)^{-1} \subseteq E \circ E^{-1} \in \E_Y$. Hence it follows from Lemma \ref{lem:E(f) is generated by graphs} that $\E(f) \subseteq \E$.

On the other hand, given $E \in \E|_{Y \times X}$, we have $E \subseteq (E \circ \Gr(f)^{-1}) \circ \Gr(f)$. Since $E \circ \Gr(f)^{-1} \in \E|_{Y \times Y} = \E_Y$, then we obtain $E \in \E(f)$. Hence $\E=\E(f)$, which implies that $f$ is coarse thanks to Corollary \ref{cor:E(f) is coarse iff f is coarse}.

Finally, if $\E(f) = \E(g)$ for coarse maps $f,g: X \to Y$, then $\Gr(f) \circ \Gr(g)^{-1} \in \E_Y$. Hence $f$ and $g$ are close.
\end{proof}

The following result shows that the notion of coarse equivalence from Definition \ref{introdefn:ce} coincides with the one from Definition \ref{defn:coarse equivalence} for unital coarse spaces.

\begin{cor}\label{cor:coarse equiv unital case}
Let $(X, \E_X)$ and $(Y, \E_Y)$ be unital coarse spaces.
\begin{enumerate}
  \item Let $f: X \to Y$ be a coarse equivalence with a coarse inverse $g: Y \to X$. Then we have $\E(f) = \E(g)$, which is a coarse correspondence both from $X$ to $Y$ and from $Y$ to $X$. In this case, we have $\E(f)|_{X \times X}= \E_X$.
  \item Let $\E$ be a coarse correspondence from $X$ to $Y$ and $Y$ to $X$. Let $f: X \to Y$ be a coarse map from Proposition \ref{prop:coarse corr unital case} such that $\E= \E(f)$, and similarly $g: Y \to X$ be a coarse map such that $\E=\E(g)$. Then $f$ is a coarse equivalence with a coarse inverse $g$.
\end{enumerate}
\end{cor}

\begin{proof}
(1) is straightforward, and hence we omit the proof. For (2), it suffices to note that $\{(f\circ g(y),y) ~|~ y\in Y\} \subseteq \Gr(f) \circ \Gr(g)  \in \E|_{Y \times Y} = \E_Y$ and similarly, $\{(g\circ f(x),x) ~|~ x\in X\} = \Gr(g) \circ \Gr(f)  \in \E|_{X \times X} = \E_X$. Hence we obtain the result.
\end{proof}

\subsection{The general case}

Now we move to the general case. Firstly, let us introduce the following notion:

\begin{defn}\label{defn:admissible collection}
Let $X$ be a set and $\L\subseteq \P(X)$. We say that $\L$ is \emph{admissible} if the following holds:
\begin{enumerate}
 \item Every finite subset of $X$ belongs to $\L$;
 \item For any $Y \in \L$ and $Z \subseteq Y$, then $Z \in \L$;
 \item For any $Y,Z \in \L$, then $Y \cup Z \in \L$.
\end{enumerate}
\end{defn}

The example of admissible collection we are interested in comes from coarse structures:


\begin{defn}\label{defn:admiss collection ass to coarse structure}
Let $(X, \E_X)$ be a coarse space. The collection
\[
\L(\E_X):=\{\r(E) ~|~ E \in \E_X\} = \{\s(E) ~|~ E \in \E_X\}
\]
is called \emph{associated to $(X,\E_X)$}, also denoted by $\L_X$, if the coarse structure is clear from the context.
\end{defn}

Clearly $\L(\E_X)$ associated to a coarse space $(X, \E_X)$ is always admissible. When $(X,d)$ is a discrete metric space and $\I \subseteq \E_d$ is an ideal (in the sense of Definition \ref{defn:ideals in coarse structure}), the definition for $\L(\I)$ above coincides with the one in Proposition \ref{prop:ideals in space}. In particular, any ideal $\L$ in $(X,d)$ (in the sense of Definition \ref{defn:ideals in space}) is admissible.

The following is straightforward, and hence we omit the proof.

\begin{lem}\label{lem:Delta_L belongs to E}
Let $(X, \E_X)$ be a coarse space, and $\L_X$ be the associated collection. For $L \subseteq X$, we have that $L \in \L_X$ \emph{if and only if} the set $\Delta_L=\{(x,x) ~|~ x\in L\}$ belongs to $\E_X$.
\end{lem}


To study the non-unital case, we need to consider a family of maps instead of a single map:

\begin{defn}\label{defn:admissible family}
Let $X$ be a set, $\L$ be an admissible collection in $\P(X)$ and $(Y,\E_Y)$ be a coarse space.
\begin{enumerate}
 \item We say that a family of maps $\F=\{f_L: L \to Y ~|~ L \in \L\}$ is \emph{admissible} if for any $L_1, L_2 \in \L$, the set $\{(f_{L_1}(x), f_{L_2}(x)) ~|~ x\in L_1 \cap L_2\}$ belongs to $\E_Y$.
 \item For an admissible family $\F=\{f_L: L \to Y ~|~ L \in \L\}$, denote
 \[
 \F^*\E_Y:=\{E \in X \times X ~|~ \exists~ L \in \L \text{ such that } E \subseteq L \times L \text{ and } (f_L \times f_L)(E) \in \E_Y\}.
 \]
\end{enumerate}
\end{defn}

We also need to consider the relation of closeness between families:
\begin{defn}\label{defn:closeness btw families}
Let $X$ be a set, $\L$ be an admissible collection in $\P(X)$ and $(Y,\E_Y)$ be a coarse space. Two admissible families $\F=\{f_L: L \to Y ~|~ L \in \L\}$ and $\G=\{g_{L} : L \to Y ~|~ L \in \L\}$ are called \emph{close} if for any $L \in \L$, the set $\{(f_L(x), g_L(x)) ~|~ x\in L\}$ belongs to $\E_Y$.
\end{defn}

\begin{lem}\label{lem:pullback}
Let $X$ be a set and $\L$ be an admissible collection in $\P(X)$. Let $(Y,\E_Y)$ be a coarse space, and $\F=\{f_L: L \to Y ~|~ L \in \L\}$ be an admissible family. Then
\begin{enumerate}
 \item For any $L \in \L$, we have $f_L(L) \in \L_Y$.
 \item $\F^*\E_Y$ is a coarse structure on $X$.
\end{enumerate}
\end{lem}

\begin{proof}
(1). Since $\L$ is admissible, the set $\{(f_L(x), f_L(x)) ~|~ x\in L\} \in \E_Y$ for any $L\in \L$. Hence we obtain $f_L(L) \in \L_Y$.

(2). Given $E_1, E_2 \in \F^*\E_Y$, assume that there exist $L_1, L_2 \in \L$ such that $E_i \subseteq L_i \times L_i$ and $(f_{L_i} \times f_{L_i})(E_i) \in \E_Y$ for $i=1,2$. Take $L=L_1 \cup L_2 \in \L$, then we have $E_1 \cup E_2 \subseteq L \times L$ and $E_1 \circ E_2 \subseteq L \times L$. Note that $F_i:=\{(f_L(x), f_{L_i}(x)) ~|~ x\in L_i\} \in \E_Y$ for $i=1,2$. Hence we have
\[
(f_L \times f_L)(E_i) \subseteq F_i \circ (f_{L_i} \times f_{L_i})(E_i) \circ F_i^{-1} \in \E_Y
\]
for $i=1,2$, which implies that $(f_L \times f_L)(E_1 \cup E_2) \in \E_Y$. Also note that
\[
(f_L \times f_L)(E_1 \circ E_2) \subseteq (f_L \times f_L)(E_1) \circ (f_L \times f_L)(E_2) \in \E_Y.
\]
The rest is trivial, and hence we omit the details.
\end{proof}

Now we introduce the notion of coarse family, which is the replacement of coarse maps in the general case.

\begin{defn}\label{defn:coarse family}
Let $(X, \E_X), (Y, \E_Y)$ be coarse spaces and $\L$ be an admissible collection in $\P(X)$. A family of admissible maps $\F=\{f_L: L \to Y ~|~ L \in \L\}$ is called \emph{coarse} if $\E_X \subseteq \F^*\E_Y$.
\end{defn}

It is clear from the definition that for a coarse family $\F=\{f_L: L \to Y ~|~ L \in \L\}$, we have $\L_X \subseteq \L$. Note that we do not require $\L_X = \L$ in general.


\begin{defn}\label{defn:coarse corr ass to family}
Let $X$ be a set and $\L$ be an admissible collection in $\P(X)$. Let $(Y,\E_Y)$ be a coarse space and $\F=\{f_L: L \to Y ~|~ L \in \L\}$ be an admissible family. Denote
\[
\E_{YX}(\F):=\{E \subseteq Y \times X ~|~ \exists~ L \in \L \text{ such that } E \subseteq Y \times L \text{ and } (\Id_Y \times f_L)(E) \in \E_Y\},
\]
and
\[
\E(\F):=\{E_X \cup E_Y \cup E_{XY} \cup E_{YX} ~|~ E_X \in \F^*\E_Y, E_Y \in \E_Y \text{ and } E_{YX} \cup E_{XY}^{-1} \in \E_{YX}(\F)\}.
\]
\end{defn}

We record the following elementary observation, whose proof is straightforward and hence omitted.
\begin{lem}\label{lem:elem ob for pullback}
Given $X, \L, (Y,\E_Y)$ and $\F$ as in Definition \ref{defn:coarse corr ass to family}, we have the following:
\begin{enumerate}
 \item For any $E \in \E_{YX}(\F)$, there exist $L \in \L$ and $B \in \L_Y$ such that $E \subseteq B \times L$.
 \item For $L_1, L_2 \in \L$ and $E \subseteq Y \times X$, if $E \subseteq Y \times L_i$ for $i=1,2$ and $(\Id_Y \times f_{L_1})(E) \in \E_Y$, then $(\Id_Y \times f_{L_2})(E) \in \E_Y$.
\end{enumerate}
\end{lem}

The following result is analogous to Lemma \ref{lem:facts for E(f)}.
\begin{lem}\label{lem:facts for E(F)}
Given $X, \L, (Y,\E_Y)$ and $\F$ as in Definition \ref{defn:coarse corr ass to family}, then $\E(\F)$ is a coarse structure on $X \sqcup Y$ generated by $\E(\F)|_{Y \times (X \sqcup Y)} = \E_Y \cup \E_{YX}(\F)$.
\end{lem}

\begin{proof}
Firstly, we show that $\E(\F)$ is a coarse structure on $X \sqcup Y$. Consider the admissible collection $\L':=\{A \sqcup B ~|~ A \in \L \text{ and } B \in \L_Y\}$ in $\P(X \sqcup Y)$. For any $L=A \sqcup B \in \L'$ with $A \in \L$ and $B \in \L_Y$, define a map $h_L: L \to Y$ which restricts to $f_A$ on $A$ and $\Id_B$ on $B$. Given $L_i=A_i \sqcup B_i$ with $A_i \in \L$ and $B_i \in \L_Y$ for $i=1,2$, we have
\[
\{(h_{L_1}(z), h_{L_2}(z)) ~|~ z\in L_1 \cap L_2\} = \{(f_{A_1}(x),f_{A_2}(x)) ~|~ x\in A_1 \cap A_2\} \cup \Delta_{B_1 \cap B_2},
\]
which belongs to $\E_Y$ due to Lemma \ref{lem:Delta_L belongs to E}. Hence $\H:=\{h_L ~|~L \in \L'\}$ is admissible.

Moreover, $E=E_X \cup E_Y \cup E_{XY} \cup E_{YX}$ belongs to $\mathcal{H}^*\E_Y$ if and only if there exists $L=A \sqcup B \in \L'$ with $A \in \L$ and $B \in \L_Y$ such that $E \subseteq L \times L$ and $(h_L \times h_L)(E) \in \E_Y$, which is equivalent to that $E \in \E(\F)$ thanks to Lemma \ref{lem:elem ob for pullback}. Hence we conclude that $\E(\F) = \mathcal{H}^*\E_Y$, which is a coarse structure on $X \sqcup Y$ due to Lemma \ref{lem:pullback}(2).

The rest of the proof is similar to that for Lemma \ref{lem:facts for E(f)}, and hence we omit the details.
\end{proof}

Consequently, we obtain the following:

\begin{cor}\label{cor:E(F) is coarse iff F is coarse}
Let $(X, \E_X), (Y, \E_Y)$ be coarse spaces, $\L$ be an admissible collection in $\P(X)$ and $\F=\{f_L: L \to Y ~|~ L \in \L\}$ be a family of admissible maps. Then $\F$ is coarse \emph{if and only if} $\E(\F)$ is a coarse correspondence. In this case, we say that $\E(\F)$ is the \emph{coarse correspondence associated to $\F$}.
\end{cor}

Similar to Lemma \ref{lem:E(f) is generated by graphs}, we have the following:

\begin{lem}\label{lem:E(F) is generated by graphs}
Given $X, \L, (Y,\E_Y)$ and $\F$ as in Definition \ref{defn:coarse corr ass to family}, then $\E(\F)$ is generated by $\E_Y$ and $\{\Gr(f_L) ~|~L\in \L\}$.
\end{lem}

\begin{proof}
Note that for any $L \in \L$, we have $\{(f_L(x), f_L(x)) ~|~ x\in L\} \in \E_Y$ and hence $\Gr(f_L) \in \E_{YX}(\F)$. Now given $E\in \E_{YX}(\F)$, there exists $L \in \L$ such that $E \subseteq Y \times L$ and $(\Id_Y \times f_L)(E) \in \E_Y$. Hence we have
$E \subseteq (E \circ \Gr(f_L)^{-1}) \circ \Gr(f_L)$, where $E \circ \Gr(f_L)^{-1} \in \E_Y$. Therefore we conclude the proof thanks to Lemma \ref{lem:facts for E(F)}.
\end{proof}

Now we are in the position to provide a detailed picture for coarse correspondence in the general case. This is a crucial step to achieve Proposition \ref{introprop:char for ce}.

\begin{prop}\label{thm:coarse corr general case}
Let $(X, \E_X), (Y, \E_Y)$ be coarse spaces, and $\E$ be a coarse correspondence from $X$ to $Y$. Then there exist an admissible collection $\L\subseteq \P(X)$ and an admissible family of maps $\F=\{f_L: L \to Y ~|~ L \in \L\}$ such that $\E = \E(\F)$. Moreover, such a family $\F$ is unique up to closeness.
\end{prop}

\begin{proof}
Set $\L:=\{\r(E) ~|~ E \in \E|_{X \times X}\}$, which is clearly admissible. For any $L \in \L$, Lemma \ref{lem:Delta_L belongs to E} shows that $\Delta_L \in \E_Y$. By Lemma \ref{lem:elementary facts for coarse corr}, there exists $E_L \in \E|_{Y \times X}$ such that $\Delta_L \subseteq E_L^{-1} \circ E_L$. Hence there exists a map $f_L: L \to Y$ such that $E_L \supseteq \Gr(f_L)$. For $L_1, L_2 \in \L$, we have
\[
\{(f_{L_1}(x), f_{L_2}(x)) ~|~ x\in L_1 \cap L_2\} \subseteq \Gr(f_{L_1}) \circ \Gr(f_{L_2})^{-1} \subseteq E_{L_1} \circ E_{L_2}^{-1} \in \E_Y,
\]
which implies that the family $\F=\{f_L: L \to Y ~|~ L \in \L\}$ is admissible. Applying Lemma \ref{lem:E(F) is generated by graphs}, we obtain that $\E(\F) \subseteq \E$.

Conversely, we claim that $\L=\{\s(F) ~|~ F \in \E|_{Y \times X}\}$. In fact, for any $F \in \E|_{Y \times X}$, note that $F^{-1} \circ F \in \E|_{X \times X}$ and $\r(F^{-1} \circ F) = \s(F) \in \L$. For any $E \in \E|_{X \times X}$, Lemma \ref{lem:elementary facts for coarse corr} implies that there exists $F \in \E|_{Y \times X}$ such that $E \subseteq F^{-1} \circ F$, and hence $\r(E) \subseteq \r(F^{-1}) = \s(F)$. This concludes the claim. Now given $E \in \E|_{Y \times X}$, the claim implies that there exists $L \in \L$ such that $E \subseteq Y \times L$. Then $E \subseteq (E \circ \Gr(f_L)^{-1}) \circ \Gr(f_L) \in \E(\F)$. Hence we obtain that $\E=\E(\F)$.

The last statement is easy to prove, and hence we omit the details.
\end{proof}

Finally, we consider the notion of coarse equivalence from Definition \ref{introdefn:ce}. Recall that a \emph{coarse equivalence} between two coarse spaces $(X, \E_X)$ and $(Y,\E_Y)$ is a coarse correspondence $\E$ from both $X$ to $Y$ and $Y$ to $X$. We say that $(X, \E_X)$ and $(Y,\E_Y)$ are \emph{coarsely equivalent} if there exists a coarse equivalence between $X$ and $Y$.

%


Now we would like to apply Proposition \ref{thm:coarse corr general case} to unpack Definition \ref{introdefn:ce} using families of maps, and our main target is to prove Proposition \ref{introprop:char for ce}. To make it more clear, we divide into two parts.

\begin{cor}\label{cor:coarse equiv general case 1}
Let $(X, \E_X), (Y, \E_Y)$ be coarse spaces, and $\E$ be a coarse equivalence between them. Then there exist coarse families $\F=\{f_L: L \to Y ~|~ L \in \L_X\}$ and $\G=\{g_{L'}: L' \to X ~|~ L' \in \L_Y\}$ such that $\E = \E(\F) = \E(\G)$ and satisfy the following:
\begin{equation}\label{EQ:closeness 1}
  \{(g_{f_L(L)} \circ f_L(x), x) ~|~ x\in L\} \in \E_X \quad \text{for any} \quad L \in \L_X,
\end{equation}
  and
\begin{equation}\label{EQ:closeness 2}
 \{(f_{g_{L'}(L')} \circ g_{L'}(y), y) ~|~ y\in L'\} \in \E_Y \quad \text{for any} \quad L' \in \L_Y.
\end{equation}
 In this case, we have $\F^*\E_Y = \E_X$ and $\G^* \E_X = \E_Y$.
\end{cor}

The proof is similar to that for Corollary \ref{cor:coarse equiv unital case} using Proposition \ref{thm:coarse corr general case} instead, and hence we omit the proof.

\begin{cor}\label{cor:coarse equiv general case 2}
Let $(X, \E_X), (Y, \E_Y)$ be coarse spaces, and $\F=\{f_L: L \to Y ~|~ L \in \L_X\}$, $\G=\{g_{L'}: L' \to X ~|~ L' \in \L_Y\}$ be coarse families satisfying (\ref{EQ:closeness 1}) and (\ref{EQ:closeness 2}). Then $\F^*\E_Y = \E_X, \G^* \E_X = \E_Y$ and $\E(\F) = \E(\G)$ is a coarse equivalence between $X$ and $Y$.
\end{cor}

\begin{proof}
Firstly, we show that $\F^*\E_Y = \E_X$. Since $\F$ is a coarse family, we have $\E_X \subseteq \F^*\E_Y$. Conversely given $E \in \F^*\E_Y$, there exists $L \in \L_X$ such that $E \subseteq L \times L$ and $(f_L \times f_L)(E) \in \E_Y$. Take $L':=f_L(L)$, which belongs to $\L_Y$ by Lemma \ref{lem:pullback}. Then $\G$ being coarse implies:
\[
F:=\{(g_{L'}\circ f_L(x), g_{L'}\circ f_L(y)) ~|~ (x,y) \in E\} = (g_{L'} \times g_{L'}) \circ (f_L \times f_L) (E) \in \E_X.
\]
Hence
\[
E \subseteq \{(g_{L'} \circ f_L(x), x) ~|~ x\in L\}^{-1} \circ F \circ \{(g_{L'} \circ f_L(y), y) ~|~ y\in L\} \in \E_X,
\]
which concludes that $\F^*\E_Y = \E_X$. Similarly, we have $\G^* \E_X = \E_Y$.

To see $\E(\F) = \E(\G)$, it suffices to show that $\E_{YX}(\F) = \E_{XY}(\G)^{-1}$. For $E \in \E_{YX}(\F)$, Lemma \ref{lem:elem ob for pullback}(1) shows that there exist $B \in L_Y$ and $L \in \L_X$ such that $E \subseteq L \times B$ and $(\Id_Y \times f_L)(E) \in \E_Y$. Taking $L':= B \cup f_L(L) \in \L_Y$, we have
\[
F':=\{(g_{L'}(y), g_{L'}f_L(x)) ~|~ (y,x) \in E\} = (g_{L'} \times g_{L'})\circ (\Id_Y \times f_L)(E)\in \E_X.
\]
Note that
\[
(g_B \times \Id_X)(E) \subseteq F_1 \circ F' \circ F_2,
\]
where $F_1:=\{(g_B(y), g_{L'}(y)) ~|~ y\in B\}$ and $F_2:=\{(g_{L'}f_L(x), x) ~|~ x\in L\}$. Since $\G$ is admissible, then $F_1 \in \E_X$. Also (\ref{EQ:closeness 2}) implies that $\{(g_{f(L)}f_L(x), x) ~|~ x\in L\} \in \E_X$. By Lemma \ref{lem:elem ob for pullback}, we obtain $F_2 \in \E_X$. Hence $(g_B \times \Id_X)(E) \in \E_X$, which means that $E \in \E_{XY}(\G)^{-1}$. Similarly, we have $\E_{XY}(\G) \subseteq \E_{YX}(\F)^{-1}$, which concludes the proof.
\end{proof}

Motivated by Corollary \ref{cor:coarse equiv general case 1} and Corollary \ref{cor:coarse equiv general case 2}, we introduce the following:

\begin{defn}\label{defn:coarse equiv using families}
Let $(X, \E_X), (Y,\E_Y)$ be coarse spaces, and $\L_X, \L_Y$ be the associated admissible families. A coarse family $\F=\{f_L: L \to Y ~|~ L \in \L_X\}$ is called a \emph{coarse equivalence} if there exists a coarse family $\G=\{g_{L'}: L' \to X ~|~ L' \in \L_Y\}$ (which is called a \emph{coarse inverse} to $\F$) satisfying (\ref{EQ:closeness 1}) and (\ref{EQ:closeness 2}).
\end{defn}

Consequently, Corollary \ref{cor:coarse equiv general case 1} and Corollary \ref{cor:coarse equiv general case 2} can be rewritten in the following form, which concludes Proposition \ref{introprop:char for ce}.

\begin{cor}\label{cor:coarse equiv general case}
Let $(X, \E_X)$ and $(Y, \E_Y)$ be coarse spaces. Then there exists a coarse equivalence $\E$ between $X$ and $Y$ in the sense of Definition \ref{introdefn:ce} \emph{if and only if} there exists a coarse equivalence $\F$ in the sense of Definition \ref{defn:coarse equiv using families}.
\end{cor}

Finally, we discuss the condition of bounded geometry, which will be used in the next section.

\begin{lem}\label{lem:bdd geom}
Let $(X, \E_X)$ and $(Y,\E_Y)$ be coarse spaces of bounded geometry, and $\E$ be a coarse equivalence between them. Then $\E$ also has bounded geometry as a coarse structure on $X \sqcup Y$.
\end{lem}

\begin{proof}
For $F \subseteq X \times Y$, $x\in X$ and $y\in Y$, denote
\[
F_x:=\{y\in Y ~|~ (x,y) \in F\} \quad \text{and} \quad F^y:=\{x\in X ~|~ (x,y) \in F \}.
\]
It suffices to show that for $F \in \E|_{X \times Y}$, the number $\sup_{x\in X} |F_x|$ and $\sup_{y\in Y} |F^y|$ are finite.
Note that
\[
F^{-1} \circ F = \bigcup_{x\in X} (F_x \times F_x)\quad \text{and} \quad F \circ F^{-1} = \bigcup_{y\in Y} (F^y \times F^y).
\]
Hence we have
\[
\sup_{x\in X} |F_x| \leq n(F^{-1} \circ F) \quad \text{and} \quad \sup_{y\in Y} |F^y| \leq n(F \circ F^{-1}).
\]
Note that $F^{-1} \circ F \in \E_X$ and $F \circ F^{-1} \in \E_Y$. Hence we conclude the proof since $\E_X$ and $\E_Y$ have bounded geometry.
\end{proof}

\section{Coarse equivalences induce Morita equivalences}\label{sec:morita for geometric}

In this section, we recall the known result that a coarse equivalence between coarse spaces induces a Morita equivalence between the associated uniform Roe algebras. As a special case, we obtain the proof for Theorem \ref{introthm} ``(2) $\Rightarrow$ (3)''. Here we provide a detailed proof since similar idea will be used later to treat ghostly ideals (see Section \ref{sec:ghostly ideals}).

Firstly, let us recall the following notion of Morita equivalence for $C^*$-algebras due to Rieffel (\cite{Rie82}, see also \cite[Definition 2.5.2]{MR3618901}).

\begin{defn}\label{defn:Morita equiv}
Let $A$ and $B$ be $C^*$-algebras. An \emph{$A$-$B$ imprimitivity bimodule} $\M$ is a Banach space that carries the structure of both a right Hilbert $B$-module with $B$-inner product $\langle \cdot, \cdot \rangle_B$ and a left Hilbert $A$-module with $A$-inner product $_A\langle \cdot, \cdot \rangle$ satisfying the following:
\begin{enumerate}
 \item Both inner products on $\M$ are \emph{full}, \emph{i.e.}, $_A{\langle \M, \M \rangle}= A$ and $\langle \M, \M \rangle_B = B$;
 \item $_A{\langle \xi, \eta \rangle} \cdot \zeta = \xi \cdot \langle \eta, \zeta \rangle_B$ for all $\xi,\eta,\zeta \in \M$;
 \item The actions of $A$ and $B$ on $\M$ commutes.
\end{enumerate}
Say that $A$ and $B$ are \emph{Morita equivalent} if such an $A$-$B$ imprimitivity bimodule $\M$ exists.
\end{defn}

Also recall from Section \ref{sec:intro} that two $C^*$-algebras $A$ and $B$ are \emph{stably isomorphic} if $A \otimes \K(\H)$ is $\ast$-isomorphic to $B \otimes \K(\H)$, where $\H$ is a separable infinite dimensional Hilbert space. It was proved in \cite{BGR77} (see also \cite[Chapter 7]{Lan95}) that if $A$ and $B$ are $\sigma$-unital (\emph{i.e.}, they admit countable approximate units), then $A$ and $B$ are Morita equivalent \emph{if and only if} they are stably isomorphic.

Let us restate Theorem \ref{introthm} ``(2) $\Rightarrow$ (3)'' in a more general form as follows:

\begin{prop}\label{prop:ce induces me}
Let $(X, \E_X)$ and $(Y, \E_Y)$ be coarse spaces of bounded geometry which are coarsely equivalent. Then the uniform Roe algebras $C^*_u(X, \E_X)$ and $C^*_u(Y, \E_Y)$ are Morita equivalent. Additionally if $\E_X$ and $\E_Y$ are unital, then $C^*_u(X, \E_X)$ and $C^*_u(Y, \E_Y)$ are stably isomorphic.
\end{prop}

This can be proved essentially by combining \cite[Corollary 3.6]{STY02} and \cite[Theorem 2.8]{MRW87} (where second countability is assumed) using the notion of Morita equivalence for groupoids.  Also note that the metric space case was proved directly in \cite[Theorem 4]{BNW07}. For the reader's convenience, here we provide a direct proof for the general case.

\begin{proof}[Proof of Proposition \ref{prop:ce induces me}]
Let $\E$ be a coarse equivalence between $(X, \E_X)$ and $(Y, \E_Y)$, which is a coarse structure on $X \sqcup Y$. For simplicity, denote $A:=C^*_u(X, \E_X)$ and $B:=C^*_u(Y, \E_Y)$. Define $\mathcal{M}$ to be the norm closure of
$$\{T \in \CC_u[X \sqcup Y, \E] ~|~ \supp(T) \subseteq X \times Y\}$$
in $\B(\ell^2(X \sqcup Y))$, which is a left $A$-module as well as a right $B$-module under composition of operators. Also define the inner products
\[
_A \langle \cdot, \cdot \rangle : \M \times \M \longrightarrow A \quad \text{by} \quad _A \langle T, S \rangle:= TS^* \quad  \text{for} \quad T,S \in \M,
\]
and
\[
 \langle \cdot, \cdot \rangle_B : \M \times \M \longrightarrow B \quad \text{by} \quad \langle T, S \rangle_B:= T^*S \quad  \text{for} \quad T,S \in \M.
\]
It is straightforward to check that $\M$ is a right Hilbert $B$-module as well as a left Hilbert $A$-module with commuting actions of $A$ and $B$ on $\M$. Moreover, we have
\[
_A{\langle T, S \rangle} \cdot R = (TS^*)R = T(S^*R) =  T \cdot \langle S, R \rangle_B
\]
for all $T,S,R \in \M$.

It remains to show that both inner products are full. Given
$a \in \CC_u[X, \E_X]$, denote $E:=\supp(a)$. It follows from Lemma
\ref{lem:elementary facts for coarse corr} that there exists
$F \in \E|_{X \times Y}$ such that $\Delta_{\s(E)}\subseteq F \circ F^{-1}$.
Hence $F^{-1}$ contains a graph of a map $\phi: \s(E) \to Y$. Thanks to Lemma
\ref{lem:bdd geom}, we can apply \cite[Lemma 4.10]{Roe03} to write $\Gr(\phi)$ as a finite union of $\Gr(\phi_{i})$ for $i=1,2,\cdots, n$ where each $\phi_{i}:D_{i}\rightarrow Y$ is an injective map. Moreover, we can further assume
that $\{D_{i}\}_{i=1}^n$ are mutually disjoint.

For each $i=1,2,\cdots, n$, we set $T_{i} = \chi_{\Gr(\phi_{i})}$, which is regarded as a bounded
operator from $\ell^2(X)$ to $\ell^2(Y)$. Note that $\supp(T_{i}) \subseteq F^{-1}$, and hence $T^*_{i} \in
\M$. Moreover, for each $i=1,2,\cdots, n$, we have
$_A{\langle aT^*_{i}, T^*_{i} \rangle} = aT^*_{i}T_{i} = a \chi_{\Delta_{D_{i}}}$.
Hence
\[
{\sum_{i=1}^n}~  {_A{\langle aT^*_{i}, T^*_{i} \rangle}} =\sum_{i} a
\chi_{\Delta_{D_{i}}} = a \chi_{\Delta_{\s(E)}} =a,
\]
which concludes that the inner $A$-product is full. Similarly, we obtain that the inner $B$-product is also full. Hence $\M$ is an $A$-$B$ imprimitivity bimodule, which concludes that $A$ and $B$ are Morita equivalent.
\end{proof}

\section{Rigidity for geometric ideals}\label{sec:rigidity for geometric ideals}

In this section and the next, we focus on the rigidity problem for geometric ideals and prove Theorem \ref{introthm} ``(1) $\Rightarrow$ (2)'', which is the main task of this paper. To make the proof more transparent, we prove a weak version in this section. The proof is relatively easier, while contains almost all the necessary ingredients to treat the general case.


Recall that for an ideal $I$ in the uniform Roe algebra of a discrete metric space $(X,d)$ of bounded geometry, the associated coarse structure is $\I(I):=\{\supp_\varepsilon(T) ~|~ T \in I, \varepsilon>0\}$ from Proposition \ref{prop:ideals in coarse strcutrue}. We aim to prove the following:

\begin{thm}\label{thm:rigidity for geom ideals}
Let $(X,d_X), (Y,d_Y)$ be discrete metric spaces of bounded geometry, and $I_X, I_Y$ be geometric ideals in the uniform Roe algebras $C^*_u(X)$ and $C^*_u(Y)$, respectively. Assume that $I_X$ and $I_Y$ are isomorphic. Then the coarse spaces $(X,\I(I_X))$ and $(Y, \I(I_Y))$ are coarsely equivalent.
\end{thm}

The proof of Theorem \ref{thm:rigidity for geom ideals} is divided into two parts. We follow the outline of \cite{BBFKVW22}, while requiring more techniques to overcome the issue of lacking units.

Throughout the rest of this section, let $(X,d_X), (Y,d_Y)$ be discrete metric spaces of bounded geometry, and $I_X, I_Y$ be geometric ideals in the uniform Roe algebras $C^*_u(X)$ and $C^*_u(Y)$, respectively. Denote the associated ideal in $\E_{d_X}$ by $\I_X:=\I(I_X)$, and the associated ideal in $(X,d_X)$ by $\L_X:=\L(\I_X)$ (from Proposition \ref{prop:ideals in coarse strcutrue} and \ref{prop:ideals in space}). Similarly, denote $\I_Y:=\I(I_Y)$ and $\L_Y:=\L(\I_Y)$. Also let $\Phi\colon I_X \to I_Y$ be an isomorphism.

\subsection{Rigid isomorphisms for geometric ideals}
\label{ssec:rigid-isom-geom}

Firstly, recall that the isomorphism $\Phi\colon I_X \to I_Y$ can always be spatially implemented. The proof is similar to that for \cite[Lemma 3.1]{SW13} based on the representation theory for compact operators together with the fact that finite subsets are entourages, and hence we omit the details.

\begin{lem}\label{lem:sp impl}
There exists a unitary $U: \ell^2(X) \to \ell^2(Y)$ such that $\Phi(a) = U a U^*$ for any $a\in I_X$.
\end{lem}

%
%

In \cite{BF21}, the authors introduced a notion of rigid isomorphism which plays a key role to attack the rigidity problem. The idea can be traced back to the work \cite{SW13}. We introduce the following to deal with the non-unital case.

\begin{defn}\label{defn:rigid isom}
With the same notation as above, the isomorphism $\Phi\colon I_X\rightarrow I_Y$ is called a \textit{rigid isomorphism} if $\Phi$ can be spatially implemented by a unitary $U: \ell^2(X) \to \ell^2(Y)$ satisfying:
  \[
    \inf_{x\in L}\sup_{y\in Y}|\langle
    U\delta_{x},\delta_{y}\rangle|>0\quad\text{for all } \quad L\in\L_{X},
   \]
and
   \[
    \inf_{y\in L'}\sup_{x\in X}|\langle
    U^{*}\delta_{y},\delta_{x}\rangle|>0\quad\text{for all } \quad L'\in\L_{Y}.
  \]
In this case, we say that $I_X$ and $I_Y$ are \emph{rigidly isomorphic}.
\end{defn}

The first step to attack the rigidity problem is the following (comparing \cite[Theorem 4.12]{BF21}):

\begin{prop}\label{prop:rigid isom}
With the same notation as above, assume that $I_X$ and $I_Y$ is rigidly isomorphic. Then the coarse spaces $(X,\I_X)$ and $(Y, \I_Y)$ are coarsely equivalent.
\end{prop}

To prove Proposition \ref{prop:rigid isom}, we need some preparations.

\begin{defn}[{\cite[Definition 4.3]{BF21}}]\label{defn:appr}
Let $(X,d)$ be a discrete metric of bounded geometry, $\varepsilon>0$ and $r \geq 0$. An operator $a\in \B(\ell^2(X))$ is called \emph{$\varepsilon$-$r$-approximable} if there exists $b\in \B(\ell^2(X))$ with propagation at most $r$ such that $\|a-b\| \leq \varepsilon$.
\end{defn}

The following key ``equi-approximability lemma'' was originally obtained in \cite[Section 4]{BF21} (see also \cite[Lemma 1.9]{BBFKVW22}), which will be used frequently in the sequel.

\begin{lem}\label{lem:equi-app}
Let $(X,d)$ be a discrete metric of bounded geometry, and $\{a_n\}_n$ be a sequence of operators such that $\SOT$-$\sum_{n\in M} a_n$ converges to an element in $C^*_u(X)$ for all $M \subseteq \NN$. Then for all $\varepsilon>0$, there exists $r>0$ such that $\SOT$-$\sum_{n\in M} a_n$ is $\varepsilon$-$r$-approximable for all $M \subseteq \NN$.
\end{lem}

We also need the following lemma, which is analogous to \cite[Lemma 4.11]{BF21}. As explained in Section \ref{sec:intro}, it is unclear whether the coarse structures $\I_X$ and $\I_Y$ are small in the sense of \cite[Definition 4.2]{BF21}. This obstructs us from using the results in \cite[Section 4]{BF21} directly, and hence more techniques are required.

\begin{lem}\label{lem:ctrl sets}
  With the notation as above, for any $E\in \I_{X}$ and $\delta>0$, the
  following set
  \[
    F(E,\delta):=\left\{(y_{1},y_{2})\in Y\times Y\;\middle\vert\;
    \begin{array}{l}
      \text{ there exists }
      (x_{1},x_{2})\in E\text{ such that}\\
      |\langle
      U\delta_{x_{1}},\delta_{y_{1}} \rangle|\geq\delta\text{ and }
       |\langle U\delta_{x_{2}},\delta_{y_{2}} \rangle|\geq\delta
    \end{array}
\right\}
\]
belongs to $\I_{Y}$.
\end{lem}

\begin{proof}
Suppose otherwise. Then there exist $E\in\I_{X}$ and $\delta>0$ such that $F(E,\delta)\notin \I_{Y}$. Hence for any $F\in\I_{Y}$, we have $F(E,\delta) \neq F$. Note that subsets of any entourage are entourages (see Definition \ref{defn:coarse structure}(3)), so we have $F(E,\delta) \setminus F \neq \emptyset$. Therefore, we can find $(x_{1}^{F},x_{2}^{F})\in E$ and $(y_{1}^{F},y_{2}^{F})\in Y\times Y$ such that $|\langle U\delta_{x_{1}^{F}},\delta_{y_{1}^{F}} \rangle|\geq\delta$ and $|\langle U\delta_{x_{2}^{F}},\delta_{y_{2}^{F}} \rangle|\geq\delta$ while $(y_{1}^{F},y_{2}^{F})\notin F$. Ordering $\I_Y$ by inclusion, it follows from \cite[Lemma 4.10]{BF21} that there exist cofinal
$I\subset \I_{Y}$, $J\subset\I_{Y}$ and a map $\varphi: I\rightarrow J$ such that
\begin{enumerate}
\item\label{item:BF411.1} $x_{1}^{F}\neq x_{1}^{F'}$ and $x_{2}^{F}\neq x_{2}^{F'}$ for
  all $F\neq F'\in J$;
\item\label{item:BF411.2} $x_{1}^{F} = x_{1}^{\varphi(F)}$ and $x_{2}^{F} = x_{2}^{\varphi(F)}$ for
  all $F\in I$.
\end{enumerate}
Since elements in $\{x_{1}^{F}\}_{F\in J}\subseteq X$ are distinct, then $J$ is countable. For $A\subseteq J$, condition (\ref{item:BF411.1}) and $\I_X$ having bounded geometry imply that $\SOT$-$\sum_{F\in A}e_{x_{1}^{F}x_{2}^{F}}\in \mathbb{C}_{u}[X,\I_{X}]$ with support in $E$. Hence $\SOT$-$\sum_{F\in A}\Phi(e_{x_{1}^{F}x_{2}^{F}}) = \SOT$-$\Phi(\sum_{F\in A}e_{x_{1}^{F}x_{2}^{F}})$ belongs to $C^*_u(Y)$ since $\Phi$ is $\SOT$-continuous. Applying Lemma \ref{lem:equi-app}, for $\epsilon = \delta^{2}/3$ there exists $r>0$ such that $\SOT$-$\sum_{F\in A}\Phi(e_{x_{1}^{F}x_{2}^{F}})$ is $\epsilon$-$r$-approximable for all $A \subseteq J$. In particular, we have
\begin{equation}\label{EQ:est 1}
\|\chi_{A}\Phi(e_{x_{1}^{F}x_{2}^{F}})\chi_{B}\|<\epsilon
\end{equation}
for all $F \in J$ and $A,B \subseteq Y$ with $d_Y(A,B) \geq r$.

By Proposition \ref{prop:ideals in space}, we can assume that $E\subseteq L\times L$ for some $L\in \L_{X}$. Since
$\{\chi_{L'}\}_{{L'\in\L_{Y}}}$ is an approximate unit for $I_Y$ due to Lemma \ref{lem:app unit}, there exist $L'\in\L_{Y}$ such that
\[
\|\Phi(\chi_{L})\chi_{L'}-\Phi(\chi_{L})\|<\epsilon \quad \text{and} \quad \|\chi_{L'}\Phi(\chi_{L})-\Phi(\chi_{L})\|<\epsilon.
\]
Thus, we have
\begin{align*}
&\|\chi_{L'}\Phi(e_{x_{1}^{F}x_{2}^{F}})\chi_{L'}  - \Phi(e_{x_{1}^{F}x_{2}^{F}})\| \leq \|\chi_{L'}\Phi(e_{x_{1}^{F}x_{2}^{F}})\chi_{L'} - \Phi(e_{x_{1}^{F}x_{2}^{F}})\chi_{L'}\| + \|\Phi(e_{x_{1}^{F}x_{2}^{F}})\chi_{L'} - \Phi(e_{x_{1}^{F}x_{2}^{F}})\| \\
& \leq \|\chi_{L'}\Phi(\chi_{L})\Phi(e_{x_{1}^{F}x_{2}^{F}}) - \Phi(\chi_{L})\Phi(e_{x_{1}^{F}x_{2}^{F}})\| + \|\Phi(e_{x_{1}^{F}x_{2}^{F}})\Phi(\chi_{L}) \chi_{L'} - \Phi(e_{x_{1}^{F}x_{2}^{F}})\Phi(\chi_{L})\| \\
& \leq 2\varepsilon.
\end{align*}

Setting $F_{1} :=\{(y_1,y_2) \in L'\times L' ~|~ d_Y(y_1, y_2) <r \}$, it is clear that $F_{1}\in\I_{Y}$. For all $A,B\subseteq Y$ which are $F_{1}$-separated, then $d_Y(L'\cap A, L'\cap B) \geq r$. Hence
\[
\|\chi_{A}\Phi(e_{x_{1}^{F}x_{2}^{F}})\chi_{B}\| \leq \|\chi_{A}\chi_{L'}\Phi(e_{x_{1}^{F}x_{2}^{F}})\chi_{L'}\chi_{B}\| + 2\varepsilon = \|\chi_{A\cap L'}\Phi(e_{x_{1}^{F}x_{2}^{F}})\chi_{L'\cap B}\| + 2\varepsilon  < 3\varepsilon = \delta^2,
\]
where the penultimate inequality comes from (\ref{EQ:est 1}). Since $I$ is cofinal in $\I_{Y}$, we can assume that $F_{1}\in I$.

Then we have
\[
 \left\|\chi_{\left\{y_{1}^{F_{1}}\right\}}\Phi(e_{x_{1}^{F_{1}}x_{2}^{F_{1}}})\chi_{\left\{y_{2}^{F_{1}}\right\}}\right\|
   = \left\|\chi_{\left\{y_{1}^{F_{1}}\right\}}\Phi(e_{x_{1}^{\varphi(F_{1})}x_{2}^{\varphi(F_{1})}})\chi_{\left\{y_{2}^{F_{1}}\right\}}\right\|<\delta^{2}.
\]
On the other hand, a direct calculation shows that
\[
\left\|\chi_{\left\{y_{1}^{F_{1}}\right\}}\Phi(e_{x_{1}^{F_{1}}x_{2}^{F_{1}}})\chi_{\left\{y_{2}^{F_{1}}\right\}}\right\| = \left|\left\langle \delta_{y_{1}^{F_{1}}},U\left(\delta_{x_{1}^{F_{1}}}\right) \right\rangle \right|\cdot \left|\left\langle
               U\left(\delta_{x_{2}^{F_{1}}}\right),\delta_{y_{2}^{F_{1}}} \right\rangle\right| \geq \delta^2,
\]
which leads to a contradiction.
\end{proof}

Now we are in the position to prove Proposition \ref{prop:rigid isom}.

\begin{proof}[Proof of Proposition \ref{prop:rigid isom}]
Let $U\colon \ell^{2}(X)\rightarrow\ell^{2}(Y)$ be a unitary operator which spatially implements the rigid isomorphism $\Phi: I_X \to I_Y$. Hence:
\begin{itemize}
\item for any $L\in\L_{X}$, there exist $\delta_{L}>0$ and $f_{L}:L\rightarrow
  Y$ such that
  \[
|\langle U\delta_{x},\delta_{f_{L}(x)} \rangle|\geq\delta_{L} \quad \text{for any} \quad x\in L;
  \]
\item for any $L'\in\L_{Y}$, there exist $\delta_{L'}>0$ and $g_{L'}:L'\rightarrow
  X$ such that
  \[
|\langle U^{*}\delta_{y},\delta_{g_{L'}(y)} \rangle|\geq\delta_{L'} \quad \text{for any} \quad y\in L'.
  \]
\end{itemize}
We aim to show that the families $\F:=\{f_L ~|~ L \in \L_X\}$ and $\G:=\{g_{L'} ~|~ L' \in \L_Y\}$ provide a coarse equivalence between $(X, \I_X)$ and $(Y, \I_Y)$ in the sense of Definition \ref{defn:coarse equiv using families}, and hence conclude the proof thanks to Corollary \ref{cor:coarse equiv general case}.

Firstly, we show that the families $\F$ and $\G$ are coarse. For $L_{1},L_{2}\in \L_{X}$, consider the set
$\{(f_{L_{1}}(x),f_{L_{2}}(x))\;|\;x\in L_{1}\cap L_{2}\}$, which is contained in $F(\Delta_{L_{1}\cap L_{2}},\delta_{L_{1}}\wedge \delta_{L_{2}})$. By Lemma \ref{lem:ctrl sets}, the set $F(\Delta_{L_{1}\cap L_{2}},\delta_{L_{1}}\wedge \delta_{L_{2}})$ belongs to $\I_{Y}$. Hence $\F$ is admissible.

Given $E\in\I_{X}$, we can assume that $E \subseteq L \times L$ for some $L \in \L_X$ thanks to Proposition \ref{prop:ideals in space}. Then for any $(x_1,x_2) \in E$, it follows from the definition that $(f_L(x_1), f_L(x_2)) \in F(E,\delta_L)$. In other words, we obtain $(f_{L}\times f_{L})(E)\subseteq F(E,\delta_{L})$, where the latter belongs to $\I_Y$ due to Lemma \ref{lem:ctrl sets}. Hence we have $\I_{X}\subseteq\F^{*}(\I_{Y})$, which concludes that the family $\F$ is coarse. Similarly, we obtain that the family $\G$ is also coarse. Moreover by Lemma \ref{lem:pullback}, we have $f_L(L) \in \L_Y$ for any $L \in \L_X$, and $g_{L'}(L') \in \L_Y$ for any $L' \in \L_Y$.

It remains to show that (\ref{EQ:closeness 1}) and (\ref{EQ:closeness 2}) holds for $\F$ and $\G$. Fix $L'\in\L_{Y}$ and denote $L'':=g_{L'}(L')$. By the choice of $g_{L'}$ and $f_{L''}$, we have
\[
  |\langle U \delta_{g_{L'}(y)},\delta_{y} \rangle| =  |\langle \delta_{g_{L'}(y)}, U^*\delta_{y} \rangle| \geq \delta_{L'} \text{ for any } y\in L'
\]
and
\[
|\langle U \delta_{g_{L'}(y)}, \delta_{f_{L''}(g_{L'}(y))} \rangle|\geq\delta_{L''} \text{ for any } y\in L'.
\]
This implies that
\[
\{(y,f_{L''}(g_{L'}(y)))\;|\; y\in L'\}\subseteq F(\Delta_{L''},\delta_{L'}\wedge\delta_{L''}),
\]
which concludes (\ref{EQ:closeness 2}). Similarly, we obtain (\ref{EQ:closeness 1}) and finish the proof.
\end{proof}

\subsection{From isomorphisms to rigid isomorphisms}\label{ssec:rigid-geom-ideals}

Now we prove that any isomorphism between geometric ideals is always a rigid isomorphism, and hence conclude the proof for Theorem \ref{thm:rigidity for geom ideals}. We will follow the outline of the proof for \cite[Theorem 1.2]{BBFKVW22} with the following extra piece:

\begin{lem}\label{lem:app unit for alg}
   With that notation as above, for any $L\in\L_{X}$ and $\epsilon>0$, there
   exist $L'\in\L_{Y}$ such that $\|\chi_{L}-\Phi^{-1}(\chi_{L'})\chi_{L}\|<\epsilon$.
\end{lem}

\begin{proof}
  By Lemma \ref{lem:app unit}, $\{\chi_{L'} ~|~L'\in \L_{Y}\}$ is an approximate unit for $I_Y$.
  Hence, $\{\Phi^{-1}(\chi_{L'}) ~|~ L'\in\L_{Y}\}$ is an approximate unit for $I_X$, which concludes the proof.
\end{proof}

\begin{prop}\label{prop:isom implies rigid isom}
With the same notation as above, assume that $I_X$ and $I_Y$ are isomorphic. Then they are rigidly isomorphic.
\end{prop}

\begin{proof}
Fixing $L\in\L_{X}$ and $\epsilon>0$, take $L'\in\L_{Y}$ as in Lemma~\ref{lem:app unit for alg}. Consider $\{\Phi^{-1}(e_{yy})\}_{y\in L'}\subseteq I_X$, which is a sequence of orthogonal projections in $C^{*}_{u}(X)$ and satisfies the condition in Lemma~\ref{lem:equi-app}. Hence, there is $r>0$ such that
\[
\Phi^{-1}(\chi_{A}) = \Phi^{-1}\big(\SOT\text{-}\sum_{y\in A} e_{yy}\big) = \SOT\text{-}\sum_{y\in A} \Phi^{-1}(e_{yy})
\]
is $\epsilon$-$r$-approximable for any $A \subseteq L'$.

  Set $N_{r}:= \sup_{x\in X}|B_{X}(x,r)|$. It follows from Lemma \ref{lem:app unit for alg} that for any $x\in L$, we have
  \begin{equation}
    \label{eq:yellowstr}
    \|\chi_{\{x\}}-\Phi^{-1}(\chi_{L'})\chi_{\{x\}}\| = \|\chi_{\{x\}}\chi_{L}-\Phi^{-1}(\chi_{L'})\chi_{\{x\}}\chi_{L}\|\leq\epsilon.
  \end{equation}
  Taking $\delta\leq\frac{\epsilon}{2N_{r}}$, we denote
  \[
    M(x,\delta) :=\{y\in L'\;|\;
    \|\Phi^{- 1}(e_{yy})\delta_{x}\|\geq\delta\}\quad\text{and}\quad
    M'(x,\delta):= L'\backslash M(x,\delta).
\]
Let $\pi:\ell^{2}(X)\rightarrow\ell^{2}(B_{X}(x,r))$ be the canonical orthogonal projection. Define
$\mu:\P(M'(x,\delta))\rightarrow\ell^{2}(B_{X}(x,r))$ by
\[
\mu(A) = \pi(\Phi^{-1}(\chi_{A})\delta_{x}),\text{ for all } A\subseteq M'(x,\delta).
\]
Then we have $\|\mu(\{y\}) \|= \|\pi(\Phi^{-1}(e_{yy})\delta_{x})\|\leq\|\Phi^{-1}(e_{yy})\delta_{x}\|<\delta$ for all $y\in M'(x,\delta)$.

Since $\frac{\mu(M'(x,\delta))}{2}$ belongs to the convex hull of the range of $\mu$ and
$\ell^{2}(B_{X}(x,r))$ has real dimension at most $2N_{r}$, \cite[Lemma 2.1]{BBFKVW22} implies that there exists $A\subseteq M'(x,\delta)$ such that
\begin{equation}\label{EQ:est2}
\left\|\mu(A)- \frac{\mu(M'(x,\delta))}{2}\right\|<2N_{r}\cdot \delta \leq\epsilon.
\end{equation}
Moreover,
\begin{align}\label{EQ:est3}
  \left\|\mu(A)- \frac{\mu(M'(x,\delta))}{2}\right\|& = \left\|\pi  \left(\left(\Phi^{-
1}(\chi_{A})- \frac{1}{2}\Phi^{- 1}(\chi_{M'(x,\delta)})\right)\delta_{x}\right)\right\|\\ \nonumber & =
\left\|\pi  \left(\left(\frac{1}{2}\Phi^{-1}(\chi_{A})- \frac{1}{2}\Phi^{-
1}(\chi_{M'(x,\delta)\backslash A})\right)\delta_{x}\right)\right\|.
\end{align}
Since $\Phi^{-1}(\chi_{A})$ and $\Phi^{-1}(\chi_{M'(x,\delta)\backslash A})$ are
$\epsilon$-$r$-approximable, the convex combination $\frac{1}{2}\Phi^{-1}(\chi_{A})-
\frac{1}{2}\Phi^{- 1}(\chi_{M'(x,\delta)\backslash A})$ is also $\epsilon$-$r$-approximable.
Hence we have
\begin{equation}\label{EQ:est4}
\left\|(1-\pi)\left(\left(\frac{1}{2}\Phi^{-1}(\chi_{A})- \frac{1}{2}\Phi^{-
  1}(\chi_{M'(x,\delta)\backslash A})\right)\delta_{x}\right)\right\|<\epsilon.
\end{equation}
Using the fact that $\Phi^{-1}(\chi_{A})\Phi^{-1}(\chi_{M'(x,\delta)}) = \Phi^{-1}(\chi_{A})$ together with (\ref{EQ:est2})-(\ref{EQ:est4}), we obtain
\begin{align*}
  & \left\|\Phi^{-1}(\chi_{A})\Phi^{- 1}(\chi_{M'(x,\delta)})\delta_{x} -
    \frac{1}{2}\Phi^{- 1}(\chi_{M'(x,\delta)})\delta_{x} \right\|\\
  &\leq\left\|(1-\pi)\left(\left(\frac{1}{2}\Phi^{-1}(\chi_{A})- \frac{1}{2}\Phi^{-
  1}(\chi_{M'(x,\delta)\backslash A})\right)\delta_{x}\right)\right\|+ \left\|\mu(A)- \frac{\mu(M'(x,\delta))}{2}\right\|<2\epsilon.
\end{align*}
Hence \cite[Lemma 3.1]{BBFKVW22} implies that
$\|\Phi^{- 1}(\chi_{M'(x,\delta)})\delta_{x}\|<4\epsilon$. Moreover, note that
\[
\Phi^{-1}(\chi_{M'(x,\delta)})\delta_{x}+\Phi^{-1}(\chi_{M(x,\delta)})\delta_{x} = \Phi^{-1}(\chi_{L'})\delta_{x}.
\]
Hence combining \eqref{eq:yellowstr}, we obtain $\|\Phi^{-1}(\chi_{M(x,\delta)})\delta_{x}\|\geq1-5\epsilon$.

Taking $\epsilon=1/10$, we have
$\|\Phi^{- 1}(\chi_{M(x,\delta)})\delta_{x}\|\geq\frac{1}{2}$. Set
$\delta = \frac{1}{20N_{r}}$, then $M(x,\delta)$ is non-empty for any $x\in L$.
Hence we obtain
\[
\inf_{x\in L}\sup_{y\in L'} \|\Phi^{-1}(e_{yy})\delta_{x}\|\geq\frac{1}{20N_{r}}.
\]
Finally, note that
\[
\|\Phi^{-1}(e_{yy})\delta_{x}\|^{2} = |\langle
U^{*}\delta_{y},\delta_{x}\rangle|^{2} = |\langle U\delta_{x},\delta_{y} \rangle|^{2}.
\]
Hence $\Phi$ is a rigid isomorphism.
\end{proof}

\begin{proof}[Proof of Theorem \ref{thm:rigidity for geom ideals}]
By Proposition \ref{prop:isom implies rigid isom}, we know that $\Phi$ is a rigid isomorphism. Hence applying Proposition \ref{prop:rigid isom}, we conclude the proof.
\end{proof}

\section{Rigidity for stable geometric ideals}\label{sec:stable}

Now we move to the case of stable isomorphism and finish the proof of Theorem \ref{introthm} ``(1) $\Rightarrow$ (2)''. Recall that given a coarse space $(X,\E)$ of bounded geometry, the \emph{stable uniform Roe algebra} of $(X,\E)$ is defined to be
\[
C^*_s(X,\E):=C^*_u(X,\E) \otimes \K(\H),
\]
where $\H$ is a separable infinite dimensional Hilbert space. Note that $C^*_s(X,\E)$ can be regarded as a $C^*$-subalgebra of $\B(\ell^2(X;\H))$. Similar to the case of the uniform Roe algebra, there is a dense $\ast$-subalgebra $\CC_s[X,\E]$ in $C^*_s(X,\E)$ under this viewpoint. More precisely, an operator $T\in \B(\ell^2(X;\H))$ belongs to $\CC_s[X,\E]$ \emph{if and only if} its support (similar to the definition in Section \ref{ssec:uniform Roe alg}) belongs to $\E$ and there exists a finite-dimensional subspace $\H' \subseteq \H$ such that each matrix entry $T(x,y)$ belongs to $\B(\H')$.


Hence Theorem \ref{introthm} ``(1) $\Rightarrow$ (2)'' can be rewritten as follows:

\begin{thm}\label{thm:rigidity for stable geom ideals}
Let $(X,d_X), (Y,d_Y)$ be discrete metric spaces of bounded geometry, and $I_X, I_Y$ be geometric ideals in the uniform Roe algebras $C^*_u(X)$ and $C^*_u(Y)$, respectively. Assume that the stable geometric ideals $C^*_s(X,\I(I_X))$ and $C^*_s(Y,\I(I_Y))$ are isomorphic. Then the coarse spaces $(X,\I(I_X))$ and $(Y, \I(I_Y))$ are coarsely equivalent.
\end{thm}

The proof is similar to that for Theorem \ref{thm:rigidity for geom ideals} but more technical, which follows the outline of that for \cite[Theorem 4.1]{BBFKVW22}. Again throughout the rest of this section, let $(X,d_X), (Y,d_Y)$ be discrete metric spaces of bounded geometry, and $I_X, I_Y$ be geometric ideals in the uniform Roe algebras $C^*_u(X)$ and $C^*_u(Y)$, respectively. Denote $\I_X:=\I(I_X), \I_Y:=\I(I_Y)$ and $\L_X:=\L(\I_X), \L_Y:=\L(\I_Y)$.
Also let $\Phi\colon C^*_s(X,\I(I_X)) \to C^*_s(Y,\I(I_Y))$ be an isomorphism with the inverse $\Psi:=\Phi^{-1}$.

Firstly, we have the following analogue of Proposition \ref{prop:isom implies rigid isom}.

\begin{prop}\label{prop:stable isom implies rigid isom}
With the same notation as above, for any unit vector $\xi \in \H$ the following holds:
\begin{enumerate}
 \item for any $L \in \L_X$, there exist $f_L: L \to Y$ and a finite-rank projection $p_L$ on $\H$ such that
 \[
 \inf_{x\in L} \| \Phi(\chi_{\{x\}} \otimes p_\xi)(\chi_{\{f_{L}(x)\}} \otimes p_L) \| > 0;
 \]
 \item for any $L' \in \L_Y$, there exist $g_{L'}: L' \to X$ and a finite-rank projection $p_{L'}$ on $\H$ such that
 \[
 \inf_{y\in L'} \| \Phi(\chi_{\{y\}} \otimes p_\xi)(\chi_{\{g_{L'}(y)\}} \otimes p_{L'}) \| > 0.
 \]
\end{enumerate}
Here $p_\xi$ is the orthogonal projection from $\H$ onto $\CC\xi$.
\end{prop}

The proof is similar to that for \cite[Theorem 4.1]{BBFKVW22} (with the same idea presented in the proof of Proposition \ref{prop:isom implies rigid isom}), and hence we only provide a sketch here.
\begin{proof}[Sketch of proof for Proposition \ref{prop:stable isom implies rigid isom}]
Note that the set
\[
\{\chi_{L'} \otimes q ~|~ L' \in \L_Y, q \text{ is a finite dimensional projection on } \H\}
\]
is an approximate unit for $I_Y \otimes \K(\H)$. Hence given $\varepsilon>0$ and $L \in \L_X$, there exists $L' \in \L_Y$ and a finite dimensional projection $p_L$ on $\H$ such that
\[
\|\Psi(\chi_{L'} \otimes p_L)(\chi_L \otimes p_\xi) - \chi_L \otimes p_\xi\| < \varepsilon.
\]
Therefore, we have
\[
\|(\Id_{\ell^2(L';\H)} - \chi_{L'}\otimes p_L) \Phi(\chi_{\{x\}} \otimes \xi)\| < \varepsilon, \quad \forall x\in L.
\]
This provides a similar condition as in the hypothesis of \cite[Lemma 4.3]{BBFKVW22}, which allows us to apply the same argument therein to obtain a constant $\delta=\delta(\varepsilon, L)>0$ and a map $f_L: L \to L'$ satisfying
\[
\|\Phi(\chi_{\{x\}} \otimes p_\xi) (\chi_{\{f_{L}(x)\}} \otimes p_L)\| = \|\Psi(\chi_{\{f_{L}(x)\}} \otimes p_L)(\delta_x \otimes \xi)\| > \delta, \quad \forall x\in L.
\]
Hence we obtain (1), and (2) can be treated similarly.
\end{proof}

Now we show that the conditions in Proposition \ref{prop:stable isom implies rigid isom} imply the required coarse equivalence, analogous to Proposition \ref{prop:rigid isom}.

\begin{prop}\label{prop:stable rigid isom}
With the same notation as above, assume that for any unit vector $\xi \in \H$, condition (1) and (2) in Proposition \ref{prop:stable isom implies rigid isom} hold with functions $f_L$ and $g_{L'}$ for $L \in \L_X$ and $L' \in \L_Y$. Then the family $\F:=\{f_L~|~ L \in \L_X\}$ is a coarse equivalence (in the sense of Definition \ref{defn:coarse equiv using families}) from $(X,\I_X)$ to $(Y,\I_Y)$ with a coarse inverse $\G:=\{g_{L'} ~|~ L' \in \L_Y\}$.
\end{prop}

To prove Proposition \ref{prop:stable rigid isom}, first note that there exists a unitary $U: \ell^2(X) \otimes \H \to \ell^2(Y) \otimes \H$ such that $\Phi(a) = U a U^*$ for any $a\in I_X \otimes \K(\H)$, which is similar to Lemma \ref{lem:sp impl}. Analogous to Lemma \ref{lem:ctrl sets}, we also need the following lemma:

\begin{lem}\label{lem:stable ctrl sets}
  With the notation as above, for any $E\in \I_{X}, \delta>0$ and a finite dimensional subspace $V \subseteq \H$, the
  following set
  \[
    F(E,\delta,V):=\left\{(y_{1},y_{2})\in Y\times Y\;\middle\vert\;
    \begin{array}{l}
      \text{ there exist }
      (x_{1},x_{2})\in E, \text{unit vectors } v_1, v_2 \in V \\ \text{ and unit vectors }w_1, w_2 \in \H   \text{ such that} \\
      |\langle
      U(\delta_{x_{1}}\otimes v_1),\delta_{y_{1}} \otimes w_1 \rangle|\geq\delta\text{ and }
       |\langle U(\delta_{x_{2}}\otimes v_2),\delta_{y_{2}} \otimes w_2 \rangle|\geq\delta
    \end{array}
\right\}
\]
belongs to $\I_{Y}$.
\end{lem}

The proof for Lemma \ref{lem:stable ctrl sets} is similar to that for Lemma \ref{lem:ctrl sets} with minor modifications, and hence we only provide a sketch to explain the difference.

\begin{proof}[Sketch of proof for Lemma \ref{lem:stable ctrl sets}]
Suppose otherwise. Then there exist $E\in\I_{X}, \delta>0$ and a finite dimensional subspace $V \subseteq \H$ such that $F(E,\delta,V) \notin \I_{Y}$. Therefore for any $F\in\I_{Y}$, we can find $(x_{1}^{F},x_{2}^{F})\in E, (y_{1}^{F},y_{2}^{F})\in Y\times Y$, unit vectors $v_1^F, v_2^F \in V$ and unit vectors $w_1^F, w_2^F \in \H$ such that
\[
\left|\left\langle U\left(\delta_{x^F_{i}}\otimes v^F_i\right),\delta_{y^F_{i}} \otimes w^F_i \right\rangle\right| \geq \delta
\]
for $i=1,2$, while $(y_{1}^{F},y_{2}^{F})\notin F$. Note that $v_1^F, v_2^F$ are unit vectors in $V$ and $V$ is finite dimensional. Hence after a small perturbation, we can assume that the set $\left\{v_1^F, v_2^F\right\}_{F \in \I_Y}$ is finite. Ordering $\I_Y$ by inclusion, it follows from \cite[Lemma 4.10]{BF21} that there exist cofinal
$I\subset \I_{Y}$, $J\subset\I_{Y}$ and a map $\varphi: I\rightarrow J$ such that
\begin{enumerate}
\item $x_{1}^{F}\neq x_{1}^{F'}$ and $x_{2}^{F}\neq x_{2}^{F'}$ for
  all $F\neq F'\in J$;
\item $x_{1}^{F} = x_{1}^{\varphi(F)}$, $x_{2}^{F} = x_{2}^{\varphi(F)}$, $v_{1}^{F} = v_{1}^{\varphi(F)}$ and $v_{2}^{F} = v_{2}^{\varphi(F)}$ for all $F\in I$.
\end{enumerate}
The rest is similar to that of Lemma \ref{lem:ctrl sets}, and hence we omit the details.
\end{proof}

Now we use Lemma \ref{lem:stable ctrl sets} to prove Proposition \ref{prop:stable rigid isom}.

\begin{proof}[Proof of Proposition \ref{prop:stable rigid isom}]
As remarked above, there exists a unitary $U: \ell^2(X) \otimes \H \to \ell^2(Y) \otimes \H$ such that $\Phi(a) = U a U^*$ for any $a\in I_X \otimes \K(\H)$. Fix a unit vector $\xi \in \H$. Direct calculations show that conditions in Proposition \ref{prop:stable isom implies rigid isom} can be translated as follows:
\begin{itemize}
 \item for any $L \in \L_X$, there exist $\delta_L>0$, $f_L: L \to Y$ and a finite-rank projection $p_L$ on $\H$ such that for any $x\in L$, there exists a unit vector $w_x \in p_L \H$ satisfying:
 \[
 |\langle U(\delta_x \otimes \xi), \delta_{f_L(x)} \otimes w_x \rangle| \geq \delta_L;
 \]
 \item for any $L' \in \L_Y$, there exist $\delta_{L'}>0$, $g_{L'}: L' \to X$ and a finite-rank projection $p_{L'}$ on $\H$ such that for any $y\in L'$, there exists a unit vector $v_y \in p_{L'} \H$ satisfying:
 \[
 |\langle U^*(\delta_y \otimes \xi), \delta_{g_{L'}(y)} \otimes v_y \rangle| \geq \delta_{L'}.
 \]
\end{itemize}
We aim to show that the families $\F:=\{f_L ~|~ L \in \L_X\}$ and $\G:=\{g_{L'} ~|~ L' \in \L_Y\}$ provide a coarse equivalence between $(X, \I_X)$ and $(Y, \I_Y)$, and hence conclude the proof thanks to Corollary \ref{cor:coarse equiv general case}.

Firstly, we show that the families $\F$ and $\G$ are coarse. For $L_{1},L_{2}\in \L_{X}$, consider the set
$\{(f_{L_{1}}(x),f_{L_{2}}(x))\;|\;x\in L_{1}\cap L_{2}\}$, which is contained in $F(\Delta_{L_{1}\cap L_{2}},\delta_{L_{1}}\wedge \delta_{L_{2}}, \CC \xi)$. By Lemma \ref{lem:stable ctrl sets}, the set $F(\Delta_{L_{1}\cap L_{2}},\delta_{L_{1}}\wedge \delta_{L_{2}}, \CC \xi)$ belongs to $\I_{Y}$. Hence $\F$ is admissible.

Given $E\in\I_{X}$, we can assume that $E \subseteq L \times L$ for some $L \in \L_X$ thanks to Proposition \ref{prop:ideals in space}. Then for any $(x_1,x_2) \in E$, it follows that $(f_L(x_1), f_L(x_2)) \in F(E,\delta_L, \CC \xi)$. Hence we have $\I_{X}\subseteq\F^{*}(\I_{Y})$, which concludes that the family $\F$ is coarse. Similarly, we obtain that the family $\G$ is also coarse.

It remains to show that (\ref{EQ:closeness 1}) and (\ref{EQ:closeness 2}) holds for $\F$ and $\G$. Fix $L'\in\L_{Y}$ and denote $L'':=g_{L'}(L')$. By the choice of $g_{L'}$ and $f_{L''}$, for any $y\in L'$ there exists $v_y \in p_{L'} \H$ such that
\[
  |\langle U (\delta_{g_{L'}(y)} \otimes v_y),\delta_{y}\otimes \xi \rangle| \geq \delta_{L'}
\]
and there exists $w_{g_{L'}(y)} \in p_{g_{L'}(y)}\H$ such that
\[
|\langle U (\delta_{g_{L'}(y)} \otimes \xi), \delta_{f_{L''}(g_{L'}(y))} \otimes w_{g_{L'}(y)} \rangle|\geq\delta_{L''}.
\]
This implies that
\[
\{(y,f_{L''}(g_{L'}(y)))\;|\; y\in L'\}\subseteq F(\Delta_{L''},\delta_{L'}\wedge\delta_{L''}, \CC \xi + p_{L'}\H) \in \I_Y,
\]
which concludes (\ref{EQ:closeness 2}). Similarly, we obtain (\ref{EQ:closeness 1}) and finish the proof.
\end{proof}

Combining Proposition \ref{prop:stable isom implies rigid isom} and Proposition \ref{prop:stable rigid isom}, we conclude the proof of Theorem \ref{thm:rigidity for stable geom ideals}.

Finally, we study the $\sigma$-unitalness of geometric ideals and prove the last sentence in Theorem \ref{introthm}. Recall from \cite[Definition 7.15]{WZ23} that an ideal $\L$ in a metric space $(X,d)$ is \emph{countably generated} if there exists a countable subset $\S \subseteq \L$ such that $\L$ is generated by $\S$.

\begin{lem}\label{lem:sigma unital}
Let $(X,d)$ be a discrete metric space of bounded geometry and $\L$ be a countably generated ideal in $(X,d)$. Then the geometric ideal $C^*_u(X,\I(\L))$ is $\sigma$-unital.
\end{lem}

\begin{proof}
Since $\L$ is countably generated, it follows from \cite[Lemma 7.17]{WZ23} that there exists a countable subset $\{Y_1, Y_2, \cdots, Y_n, \cdots\}$ of $\L$ such that $\L=\{Z \subseteq X ~|~ \exists~ n\in \NN \text{ such that }Z \subseteq Y_n\}$. It is then easy to see that $\{\chi_{Y_n} ~|~ n\in \NN\}$ is an approximate unit for $C^*_u(X,\I(\L))$, \emph{i.e.}, $C^*_u(X,\I(\L))$ is $\sigma$-unital.
\end{proof}

Combining Proposition \ref{prop:ce induces me}, Theorem \ref{thm:rigidity for stable geom ideals} and Lemma \ref{lem:sigma unital}, we finally conclude the proof for Theorem \ref{introthm}.


\section{Discussions on ghostly ideals}\label{sec:ghostly ideals}

In \cite{WZ23}, Wang and the second-named author introduced a notion of ghostly ideals in uniform Roe algebras. Here we provide some discussions on the rigidity problem for ghostly ideals and pose some open questions. Recall the following:

\begin{defn}[\cite{WZ23}]\label{defn:ghostly ideals}
Let $(X,d)$ be a discrete metric space of bounded geometry, and $\L$ be an ideal in $(X,d)$. The \emph{associated ghostly ideal of $\L$}, denoted by $\tilde{I}(\L)$, is defined as follows:
\[
\tilde{I}(\L):=\{T \in C^*_u(X) ~|~ \forall~\varepsilon>0, \r(\supp_\varepsilon(T))  \in \L\}.
\]
\end{defn}

Similar to the situation for geometric ideals, we would also like to study isomorphisms between ghostly ideals. More precisely, similar to the discussions in Section \ref{sec:morita for geometric} to Section \ref{sec:stable}, we pose the following questions:

\begin{question}\label{Q}
Let $(X,d_X), (Y,d_Y)$ be discrete metric spaces of bounded geometry, and $\L_X, \L_Y$ be ideals in $(X,d_X)$ and $(Y,d_Y)$, respectively.
\begin{enumerate}
 \item If $(X, \I(\L_X))$ and $(Y, \I(\L_X))$ are coarsely equivalent (in the sense of Definition \ref{introdefn:ce}), are $\tilde{I}(\L_X)$ and $\tilde{I}(\L_Y)$ Morita equivalent?
 \item Conversely, if $\tilde{I}(\L_X)$ and $\tilde{I}(\L_Y)$ are isomorphic (or stably isomorphic), are $(X, \I(\L_X))$ and $(Y, \I(\L_X))$ coarsely equivalent?
\end{enumerate}
\end{question}

Unfortunately, currently we are unable to completely answer either of these questions. We merely manage to provide a partial answer to Question \ref{Q}(1). To state our result, we need some preparations.

Let $f:(X,d_X) \to (Y,d_Y)$ be a coarse equivalence with a coarse inverse $g: Y \to X$, and $\L$ be an ideal in $(X,d_X)$. Denote $f_\ast(\L)$ the ideal in $(Y,d_Y)$ generated by the set $\{f(L) ~|~ L \in \L\}$. Then we have:

\begin{lem}\label{lem:induced by f}
With the same notation as above, we have $g_\ast f_\ast (\L) = \L$. Hence in this case, the maps $f_\ast$ and $g_\ast$ provide a one-to-one correspondence between ideals in $(X,d_X)$ and those in $(Y,d_Y)$.
\end{lem}

\begin{proof}
Firstly, note that $gf(L) \in \L$ for any $L\in \L$, which implies that $\L \subseteq g_\ast f_\ast (\L)$. Conversely, it suffices to show that $g(L') \in \L$ for any $L' \in f_\ast(\L)$. Given $L' \in f_\ast(\L)$, It follows from \cite[Lemma 7.18]{WZ23} that there exist $R>0$ and $L \in \L$ such that $L' \subseteq \Nd_R(f(L))$. Hence $g(L') \subseteq g(\Nd_R(f(L)))$, which belongs to $\L$ since $g$ is coarse. So we conclude the proof.
\end{proof}

As a direct corollary, we obtain the following:
\begin{cor}\label{cor:restricted}
Let $f:(X,d_X) \to (Y,d_Y)$ be a coarse equivalence and $\L$ be an ideal in $(X,d_X)$. Then the family $\F(f):=\{f|_L: L \to Y ~|~ L \in \L\}$ is a coarse equivalence (in the sense of Definition \ref{defn:coarse equiv using families}) between $(X, \I(\L))$ and $(Y, \I(f_\ast(\L)))$.
\end{cor}

Now we present our partial answer to Question \ref{Q}(1):

\begin{prop}\label{prop:Morita for ghost}
Let $(X,d_X), (Y,d_Y)$ be discrete metric spaces of bounded geometry, and $\L_X, \L_Y$ be ideals in $(X,d_X)$ and $(Y,d_Y)$, respectively. If $f: (X,d_X) \to (Y,d_Y)$ is a coarse equivalence with $f_\ast(\L_X) = \L_Y$, then $\tilde{I}(\L_X)$ and $\tilde{I}(\L_Y)$ are Morita equivalent.
\end{prop}

\begin{rem}
By Corollary \ref{cor:restricted}, the assumption in Proposition \ref{prop:Morita for ghost} implies that $(X, \I(\L_X))$ and $(Y, \I(\L_Y))$ are coarsely equivalent, which is the assumption considered in Question \ref{Q}(1). However, it is unclear to us whether the converse holds or not.
\end{rem}

\begin{proof}[Proof of Proposition \ref{prop:Morita for ghost}]
Let $f:(X,d_X) \to (Y,d_Y)$ be a coarse equivalence and $\E(f)$ the associated coarse correspondence. Similar to the proof of Proposition \ref{prop:ce induces me}, we define $\mathcal{M}(f)$ to be the norm closure of
$$\{T \in \CC_u[X \sqcup Y, \E(f)] ~|~ \supp(T) \subseteq X \times Y\}$$
in $\B(\ell^2(X \sqcup Y))$, which is a left $C^*_u(X)$-module as well as a right $C^*_u(Y)$-module under composition of operators. Writing $\tilde{I}_X:=\tilde{I}(\L_X)$ and $\tilde{I}_Y:=\tilde{I}(\L_Y)$, we set $\M:=\tilde{I}_X\M(f)\tilde{I}_Y$. Also define the inner products
\[
_{\tilde{I}_X} \langle \cdot, \cdot \rangle : \M \times \M \longrightarrow \tilde{I}_X \quad \text{by} \quad _{\tilde{I}_X} \langle T, S \rangle:= TS^* \quad  \text{for} \quad T,S \in \M,
\]
and
\[
 \langle \cdot, \cdot \rangle_{\tilde{I}_Y} : \M \times \M \longrightarrow \tilde{I}_Y \quad \text{by} \quad \langle T, S \rangle_{\tilde{I}_Y}:= T^*S \quad  \text{for} \quad T,S \in \M.
\]
It is straightforward to check that $\M$ is a right Hilbert $\tilde{I}_Y$-module as well as a left Hilbert $\tilde{I}_X$-module with commuting actions of $\tilde{I}_X$ and $\tilde{I}_Y$ on $\M$. We will show that $\M$ provides a Morita equivalence between $\tilde{I}_X$ and $\tilde{I}_Y$.

Firstly, we claim that $\M(f)\tilde{I}_Y \M(f)^* \subseteq \tilde{I}_X$ where $\M(f)^*:=\{T^* ~|~ T \in \M(f)\}$. Given $T_1, T_2 \in \CC_u[X \sqcup Y, \E(f)]$ with support in $X \times Y$ and $S \in \tilde{I}_Y$, there exists $R>0$ such that
\[
\supp(T_i) \subseteq \{(x,y) \in X \times Y ~|~ d_Y(f(x),y) \leq R\} \quad \text{for} \quad i=1,2.
\]
Hence for $x_1, x_2 \in X$, we have
\[
(T_1 S T^*_2)(x_1,x_2)= \sum_{\substack{y_1\in B_Y(f(x_1),R)\\y_2 \in B_Y(f(x_2,R))}} T_1(x_1,y_1) S(y_1,y_2) T_2^*(y_2,x_2).
\]
Denote $N_R:=\sup_{y\in Y} |B_Y(y,R)| < \infty$. Then for any $\varepsilon>0$ and $(x_1, x_2) \in \supp_\varepsilon(T_1 S T^*_2)$, there exist $y_1 \in B_Y(f(x_1),R)$ and $y_2 \in B_Y(f(x_2),R)$ such that
\[
(y_1,y_2) \in \supp_{\varepsilon'}(S) \quad \text{for}\quad \varepsilon':=\frac{\varepsilon}{\|T_1\| \cdot \|T_2\| \cdot N_R^2}.
\]
Since $S \in \tilde{I}_Y$, there exists $L \in \L_Y$ such that $\r(\supp_{\varepsilon'}(S)) \subseteq L$. Hence we obtain
\[
f\big(\r(\supp_\varepsilon(T_1 S T^*_2))\big) \subseteq \Nd_R\big(\r(\supp_{\varepsilon'}(S))\big) \subseteq \Nd_R(L) \in \L_Y.
\]
Therefore by Lemma \ref{lem:induced by f} and the assumption that $f_\ast(\L_X) = \L_Y$, we obtain that
\[
\r(\supp_\varepsilon(T_1 S T^*_2)) \in \L_X,
\]
which implies that $T_1 S T^*_2 \in \tilde{I}_X$ and concludes the claim.

Now we claim that $\tilde{I}_X\M(f) = \M(f)\tilde{I}_Y$. Taking a coarse inverse
$g: Y \to X$ to $f$, we decompose $\Gr(g)$ into a finite union of $\Gr(g_i)$ for $i=1,\cdots,n$ such that each $g_i: D_i \to X$
is an injective map with $\{D_i\}_{i=1}^n$ mutually disjoint. For each $i=1,\cdots, n$, define $T'_{i}:=\chi_{\Gr(g_{i})}$, which is regarded as a bounded operator
from $\ell^2(Y)$ to $\ell^2(X)$. Clearly, we have $T'_{i} \in \M(f)$,
$(T'_{i})^*T'_{i} = \Id_{\ell^2(D_{i})}$ and $\sum_{i=1}^n (T'_{i})^*T'_{i}= \Id_{\ell^2(Y)}$. Hence for $T \in \M(f)$ and $S \in \tilde{I}_Y$, we have
\[
TS = TS\sum_{i=1}^n (T'_{i})^*T'_{i} \in \big(\M(f)\tilde{I}_Y \M(f)^*\big) \cdot \M(f) \subseteq \tilde{I}_X \M(f),
\]
where the last containment follows from the previous paragraph. By symmetry, we conclude $\tilde{I}_X\M(f) = \M(f)\tilde{I}_Y$. Consequently, we obtain $\M = \tilde{I}_X\M(f) = \M(f)\tilde{I}_Y$.

From the proof of Proposition \ref{prop:ce induces me} together with Corollary \ref{cor:restricted}, we know that $ \langle \M(f), \M(f) \rangle_{C^*_u(Y)} = C^*_u(Y)$. Hence we have
\[
 \langle \M, \M \rangle_{\tilde{I}_Y} =  \langle \M(f)\tilde{I}_Y, \M(f)\tilde{I}_Y \rangle_{\tilde{I}_Y} = \tilde{I}_Y \cdot  \langle \M(f), \M(f) \rangle_{C^*_u(Y)} \cdot \tilde{I}_Y= \tilde{I}_Y.
\]
Hence the inner $\tilde{I}_Y$-product on $\M$ is full. Similarly, we obtain that the inner $\tilde{I}_X$-product on $\M$ is full. Therefore, $\M$ is a $\tilde{I}_X$-$\tilde{I}_Y$ imprimitivity bimodule, which concludes that $\tilde{I}_X$ and $\tilde{I}_Y$ are Morita equivalent.
\end{proof}

Finally, we remark that it seems that the proof of Proposition \ref{prop:Morita for ghost} does not work if we only know that $(X,\I(\L_X))$ and $(Y,\I(\L_Y))$ are coarsely equivalent instead of requiring a \emph{global} coarse equivalence from $(X,d_X)$ to $(Y,d_Y)$.

\bibliographystyle{plain}
\bibliography{bib_ideals}

\end{document}